\documentclass{amsart}

\numberwithin{equation}{section}

\usepackage{amssymb}
\usepackage{enumerate, xspace}
\usepackage{marvosym} 
\usepackage[sans]{dsfont}        
\usepackage{changebar}
\usepackage[backgroundcolor=blue!20!white, linecolor=blue!20!white,
textsize=footnotesize]{todonotes}
\usepackage[colorlinks]{hyperref}

\newtheorem{thm}{Theorem}[section]
\newtheorem{lem}[thm]{Lemma}
\newtheorem{sub-lem}[thm]{Sub-Lemma}

\newtheorem{sublem}[thm]{Sub-lemma}

\newtheorem{prob}[thm]{\bf Problem}

\newtheorem*{hyp}{Hypotheses}

\newtheorem{rem}[thm]{Remark}

\newcommand\cB{{\mathcal B}}
\newcommand\cC{{\mathcal C}}

\newcommand\cG{{\mathcal G}}

\newcommand\cL{{\mathcal L}}
\newcommand\cO{{\mathcal O}}

\newcommand\cN{{\mathcal N}}
\newcommand\cS{{\mathcal S}}

\newcommand\cU{{\mathcal U}}

\newcommand\bA{{\mathbb A}}

\newcommand\bC{{\mathbb C}}

\newcommand\bE{{\mathbb E}}
\newcommand\bG{{\mathbb G}}
\newcommand\bN{{\mathbb N}}

\newcommand\bP{{\mathbb P}}

\newcommand\bR{{\mathbb R}}
\newcommand\bT{{\mathbb T}}
\newcommand\bV{{\mathbb V}}
\newcommand\bZ{{\mathbb Z}}

\newcommand\ve{\varepsilon}

\newcommand\vf{\varphi}

\newcommand\Id{{\mathds{1}}}

\newcommand{\Leb}{\operatorname{Leb}}
\newcommand{\Const}{C_{\#}}

\newcommand{\supp}{\operatorname{supp}}
\newcommand{\Tr}{\operatorname{Tr}}

\newcommand\BV{{W^{1,1}}}
\newcommand\kd{\langle k\rangle}
\newcommand{\pP}{{\bf P}^{{\hskip -0.6 pt 1}}\!(\bR)}
\newcommand{\pf}{{\boldsymbol \vf}}
\begin{document}

\title[Transfer operators and uniformly hyperbolic map]{Statistical properties of uniformly hyperbolic maps and transfer operators' spectrum}
\author{Carlangelo Liverani}
\address{Carlangelo Liverani\\
Dipartimento di Matematica\\
II Universit\`{a} di Roma (Tor Vergata)\\
Via della Ricerca Scientifica, 00133 Roma, Italy.}
\email{{\tt liverani@mat.uniroma2.it}}
\date{\today}
\begin{abstract}
This is a lightning introduction to some modern techniques used in the study of the statistical properties of hyperbolic dynamical systems. The emphasis is not in presenting a comprehensive theory but rather in fleshing out the main ideas in the simplest and fastest possible manner so that the reader can quickly get the intuition necessary to easily read the more technical (and more complete) accounts of the theory.
\end{abstract}
\thanks{This text is an evolution of the Lectures given at the {\em International Conference on Statistical Properties of Non-equilibrium Dynamical Systems}, SUSTC, Shenzhen, July 27 - August 2, 2016 and the Lectures given at the TMU-ICTP School, Tehran, May 5-10, 2018.
{\em An introduction to the statistical properties of hyperbolic dynamical systems}. The author acknowledges the MIUR Excellence Department Project awarded to the Department of Mathematics, University of Rome Tor Vergata, CUP E83C18000100006.}
\maketitle

\tableofcontents

\section{Introduction}
This note is dedicated to presenting some basic modern techniques used to study the statistical properties of {\em chaotic} systems. Here by {\em chaotic} I mean uniformly hyperbolic systems. That is, systems that display a strong uniform sensitivity with respect to initial conditions. 
I will stress in particular the so called {\em functional approach} but I will also provide a simple introduction to the use of {\em standard pairs}. 

The functional approach has its origin in the study of the Koopman operator \cite{Ko} (acting on $L^2$) starting, at least, with Von Neumann mean ergodic theorem \cite{vN} and further developed by the Russian school \cite{CFS}. An important development of this point of view occurred with the study of the transfer operator in symbolic dynamics by  Sinai, Ruelle and Bowen \cite{S1,S2, R1,R2,B1,B2}.

Next, the functional approach developed further thanks to the work of Lasota-Yorke \cite{LY}, Ruelle \cite{R3}, Keller \cite{keller1, HK} and, more recently, Kitaev \cite{Ki}, just to mention a few.  This has eventually lead to the current theory, which has assumed its present form starting with \cite{BKL}. 

The basic idea being to study directly the spectrum of the Ruelle transfer operator without coding the system (even though the theory can be applied also to the transfer operator of a system after inducing). In order to do so it is necessary to consider the action of the transfer operator on an appropriate Banach (or Hilbert) space or, more generally, in an appropriate topology. The non trivial part of the theory rests in the identification of the appropriate topological spaces. 

In this note we will discuss only uniformly hyperbolic systems, yet the techniques presented here are relevant also in the non uniformly hyperbolic case, although they must be supplemented with essential new ideas such as Young towers \cite{Y1}, coupling \cite{Y2, Do1, DeL2} and Operator Renewal Theory \cite{Sa1}. 

The goal of this note it to explain which properties the above mentioned Banach spaces must enjoy and to provide a guide on how to construct and adapt them to the peculiarities of the systems at hand. Also I will briefly discuss the idea of coupling in a specially simple case, but I will not provide any detail on Young towers or Operator Renewal Theory. Moreover I will not discuss anything concerning hyperbolic flows, non-uniform hyperbolicity or partial hyperbolicity \cite{BDV}.
This note is a partial update with respect to the review \cite{Li3}. For a much more in depth and technical discussion of these topics see \cite{babook, babook2}.

The plan of the exposition is as follows: I start discussing the simplest possible case, smooth expanding maps of the circle. This allows to illustrate, in the simplest possible setting, the power of the functional approach and the type of results that can be obtained once such a  machinery is in place. In particular, I will show how important properties of the system such as exponential decay of correlation, CLT, stability and linear response easily follow from the spectral properties of the transfer operator.

Next, I will discuss the case of attractors, where the need to consider spaces of distributions becomes first apparent. Then I will develop the theory for the case of toral automorphisms. This may seem a bit silly as toral automorphisms can be studied directly using Fourier series. Yet, this will allow to illustrate in the simplest possible case the main ideas of the theory (anisotropic Banach spaces and coupling).

Finally, I will collect all the ideas previously illustrated and extend them to study general uniformly hyperbolic maps.
\section{Smooth expanding maps}\label{sec:expaning}
By smooth expanding map I mean a map $f\in\cC^r(\bT,\bT)$, $r\geq 2$, such that $\inf_x |f'(x)|\geq \lambda_*>1$. Clearly $(f,\bT)$ is a topological, actually differentiable, dynamical system.
Our first goal is to view it as a measurable dynamical system, hence we need to select an invariant probability measure.

Deterministic systems often have a lot of invariant measures. In particular, to any periodic orbit is associated an invariant measure (the average along the orbit). Given such plentiful possibilities, we need a criteria to select relevant invariant measures. A common choice is to consider measures that can be obtained by pushing forward a measure absolutely continuous with respect to Lebesgue. 

More precisely, let $d\mu=h(x) dx$, $h\in L^1(\bT^1,\Leb)$ and define, for all $\vf\in\cC^0(\bT,\bR)$, the average
\[
\mu(\vf)=\int_{\bT}\vf(x) \mu(dx)
\]
and the push-forward
\[
f_*\mu(\vf)=\mu(\vf\circ f).
\] 
Note that if $\mu$ is a probability measure (i.e., $h\geq 0$ and $\mu(1)=1$), then also $f_*\mu$ is a probability measure. Then
\[
\left\{\frac 1n\sum_{k=0}^{n-1}f_*^k\mu\right\}_{n\in\bN}
\]
is a weakly compact set, hence it has accumulation points. On can easily check that such accumulation points are invariant measures for $f$, that is fixed points for $f_*$ (this is, essentially, Krylov-Bogoliubov Theorem). We would then like to study such fixed points.

A simple change of variables shows that $\frac{d(f_*\mu)}{d\Leb}=\cL h$ where
\[
\cL h(x)=\sum_{f(y)=x}\frac {h(y)}{f'(y)}.
\]
The operator $\cL$ is  called the {\em (Ruelle) transfer operator}. Of course, to properly define such an operator we must specify on which space it acts. 
Since
\[
\int |\cL h(x)| dx\leq \int \cL|h| (x) dx=\int 1\circ f (x)|h(x)| dx= \int |h(x)| dx,
\]
it follows that $\cL$ is well defined as an operator from $L^1(\bT,\Leb)$ to itself, moreover it is a contraction on $L^1(\bT,\Leb)$.  In addition, if $d\mu=h_* dx$ is an invariant measure, then
\[
h_*dx=d\mu=d f_*\mu=\cL h_* dx,
\]
that is $\cL h_*=h_*$. Conversely, if $\cL h_*=h_*$, then
\[
d\mu=h_*dx=\cL h_* dx=d f_*\mu
\]
that is $d\mu=h_*dx$ is an invariant measure. 

We have thus reduced the problem of studying the invariant measures absolutely continuous with respect to Lebesgue to the problem of studying the operator $\cL$, more precisely the eigenspace associated to the eigenvalue one. We want thus to investigate the spectral theory of the operator $\cL$.
Unfortunately, the spectrum of $\cL$ on $L^1$ turns out to be the full unit disk, a not very useful fact.

Following Lasota-Yorke, we look then at the action of $\cL$ on $W^{1,1}$:\footnote{ Recall that $g\in W^{1,1}$ if $g\in L^1$ and $g'\in L^1$.}    
 \begin{equation}\label{eq:one-der}
 \frac d{dx}\cL h=\cL\left(\frac h{f'}\right)-\cL\left(h\frac{f''}{(f')^2}\right).
 \end{equation}
The above implies the so called {\em Lasota-Yorke inequalities}
\begin{equation}\label{eq:LY}
\begin{split}
&\|\cL h\|_{L^1}\leq \|h\|_{L^1}\\
&\|(\cL h)'\|_{L^1}\leq \lambda_*^{-1}\|h'\|_{L^1}+D\|h\|_{L^1}.
\end{split}
\end{equation}
Such inequalities imply that $\cL$ is well defined as an operator from $W^{1,1}$ to itself. In addition, when acting on $W^{1,1}$ it is a quasi-compact operator (see Theorem \ref{thm:hennion} for the exact statement). That is, the spectrum $\sigma_{W^{1,1}}(\cL)\subset \{z\in\cC\;:\;|z|\leq 1\}$ while the essential spectrum is strictly smaller: $\textrm{ess-}\sigma_{W^{1,1}}(\cL)\subset  \{z\in\cC\;:\;|z|\leq \lambda_*^{-1}\}$. 

To illustrate the above facts, let us consider the special case in which the distortion $D=\|\frac{f''}{(f')^2}\|_{L^\infty}$ is small, more precisely $\lambda_*^{-1}+D<1$.

Note that, if $\Leb(h)=0$, then also $\Leb(\cL h)=0$, hence the space $\bV=\{h\in L^1\;:\;\Leb(h)=0\}$ is invariant under $\cL$. Also, if $h\in \bV$, then, since $W^{1,1}\subset\cC^0$, by the mean value theorem there must exists $x_*$ such that $h(x_*)=0$, thus
\[
\|h\|_{L^1}=\int_{\bT} |h(x)|=\int_{\bT}\int_{x_*}^x |h'(y)|\leq \|h'\|_{L^1}.
\]
Next, let us define the norm $\|h\|_{W^{1,1}}=\|h'\|_{L^1}+a\|h\|_{L^1}$ for some $a>0$ to be chosen shortly.\footnote{ Note that all such norms are equivalent, so the choice of a special value of $a$ is only a matter of convenience.} Accordingly, for $h\in \bV$, equation \eqref{eq:LY} implies
\begin{equation}\label{eq:ly-00}
\begin{split}
\|\cL h\|_{W^{1,1}}&\leq \lambda_*^{-1}\|h'\|_{L^{1}}+(D+a)\|h\|_{L^1}\leq (\lambda_*^{-1}+D+a)\|h'\|_{L^1}\\
&\leq (\lambda_*^{-1}+D+a)\|h\|_{W^{1,1}}.
\end{split}
\end{equation}
We can then choose $a$ such that $\nu:=\lambda_*^{-1}+D+a<1$, which implies that $\cL$ is a strict contraction on $\bV$, that is $\sigma_{W^{1,1}}(\cL|_\bV)\subset\{z\in\bC\;:\;|z|\leq \nu\}$.
Note that $\cL'\Leb=\Leb$, hence $1\in\sigma (\cL')$ and then $1\in\sigma (\cL)$. Thus we have that there exists $h_*\in L^1$ such that $\cL h=h_*\Leb(h)+Qh$, where $\|Q\|_{W^{1,1}}\leq \nu$ and $\Leb Q=Qh_*=0$. Hence, \eqref{eq:ly-00} implies that, for each $h\in W^{1,1}$,
\[
\left\|\cL^n h-h_*\int h\right\|_{W^{1,1}}=\left\| \cL^n\left(h-h_*\int h\right)\right\|_{W^{1,1}}\leq \nu^n \left\|h-h_*\int h\right\|_{W^{1,1}}
\]
We have just proven that $h_*(x) dx$ is the only invariant measure of $f$ absolutely continuous with respect to Lebesgue.\footnote{ To make the argument precise use that $W^{1,1}$ is dense in $L^1$.} 

As already mentioned, the above spectral decomposition, and hence the uniqueness of the invariant measure absolutely continuous with respect to Lebesgue, holds in much higher generality, in particular for each $f\in \cC^2$ such that $|f'|\geq \lambda_*>1$, due to the following theorem.\footnote{ I am not stating the Theorem in its full generality as it is not needed in the following.}
\begin{thm}[\cite{He}]\label{thm:hennion}
Let $\cB\subset\cB_w$ be two Banach spaces, $\|\cdot\|$ and $\|\cdot\|_w$ being the respective norms. In addition, let $\cL:\cB\to\cB$ be a linear operator such that there exists $M,C>0$ and $n_0\in\bN$ such that $\cL^{n_0}:\cB\to\cB_w$ is a compact operator and for each $n\in\bN$ and $v\in \cB$,
\[
\begin{split}
&\|\cL^n v\|_w\leq C M^n \|v\|_w\\
&\|\cL^n v\|\leq CM^n \lambda_*^{-n}\|v\|+C M^n\|v\|_w,
\end{split}
\]
then $\cL$ has the spectral radius bounded by $M$ and the essential spectral radius bounded by $M\lambda_*^{-1}$.
\end{thm}
\begin{rem}\label{rem:compact-ball} In the following we will mostly use the above Theorem when $M=1$. Also, the compactness of the operator (for each $n_0\in\bN$) will often follow by checking that the unit ball in $\cB$, $\{v\in\cB\;:\; \|v\|\leq 1\}$, is relatively compact in $\cB_w$. Finally, if one can prove that there exists eigenvalues outside the essential spectrum (as we have done before), then Theorem \ref{thm:hennion} implies that the operator is {\em quasi compact} (that is, the maximal part of the spectrum consists of point spectrum).
\end{rem}
Since it is not hard to show that smooth expanding maps are mixing, it follows that $\cL$ cannot have eigenvalues of modulus one different from $1$ and that $1$ is a simple eigenvalue hence $\cL$ must mix exponentially fast (see \cite{babook} for an exhaustive discussion).

\begin{prob} Derive further \eqref{eq:one-der} to obtain a Lasota-Yorke inequality with respect to the norms $W^{p,1}$, $W^{p-1,1}$, $p\leq r-1$. Show then that the essential spectral radius of $\cL$ when acting on $W^{p,1}$ is bounded by $\lambda_*^{-p}$.
\end{prob}
An interesting consequence of the above analysis is that smooth expanding maps admit a unique {\em physical measure}. A measure $\mu$ is a physical measure if there exists a measurable set $A$ (called the {\em basin of attraction}) of positive Lebesgue measure such that, for all $\vf\in\cC^0$ and $x\in A$, 
\[
\lim_{n\to\infty}\frac 1n\sum_{k=0}^{n-1} \vf\circ f^n(x)=\mu(\vf).
\]
\begin{prob}\label{prob:mix} Show that if there exists $h_*\in L^1$ such that for all $h\in L^1$ we have $\lim_{n\to\infty}\cL^n h=h_*\int h$, then $h_*(x)dx$ is the unique physical measure of the system and the basin of attraction is the all space, but for a zero Lebesgue measure set.
\end{prob}
The above problem shows that, for the uniqueness of the physical measure, the speed of convergence is immaterial. Yet, if one has estimates on the speed of convergence (as in our case), then it is possible to obtain a much more useful bound. To see this, for $\vf\in\cC^1(\bT^1,\bC)$, let us set $\hat \vf=\vf-\mu(\vf)$ and compute
\begin{equation}\label{eq:var}
\begin{split}
&\left\|\sum_{k=0}^{n-1} \vf\circ f^k(x)-n\mu(\vf)\right\|_{L^2(\mu)}^2=\sum_{k,j=0}^{n-1}\int\overline{\hat \vf}\circ f^k(x) \cdot \hat \vf\circ f^j(x)\cdot h_*(x) dx\\
&=\sum_{k=0}^{n-1}\int|\hat \vf(x)|^2\cdot h_*(x) dx+2\sum_{k>j}^{n-1}\sum_{j=0}^{n-2}\int\overline{\hat \vf}\circ f^{k-j}(x) \cdot \hat \vf(x)\cdot h_*(x) dx\\
&=n\|\hat\vf\|_{L^2(\mu)}+2\sum_{l=1}^{n-1}(n-l)\int\overline{\hat \vf}\circ f^{l}(x) \cdot \hat \vf(x)\cdot h_*(x) dx\\
&=n\left[\|\hat\vf\|_{L^2(\mu)}+2\sum_{l=1}^{\infty}\int\overline{\hat \vf}\circ f^{l}(x) \cdot \hat \vf(x)\cdot h_*(x) dx\right]\\
&\phantom{=}-2\sum_{l=n}^{\infty}\int\overline{\hat \vf}(x) \cdot \cL^l(\hat \vf\cdot h_*)(x) dx
-2\sum_{l=1}^{n-1}l\int\overline{\hat \vf}(x) \cdot \cL^l(\hat \vf\cdot h_*)(x) dx.
\end{split}
\end{equation}
Note that\footnote{ Here, and the the following, we will use $\Const$ to mean a generic constant, depending only on the choice of $f$, which value can change from one occurrence to the next.} 
\[
\begin{split}
|\cL^n(\hat \vf\cdot h_*)(x)|&\leq h_*(x)\left|\int (\hat \vf\cdot h_*)(x)dx\right|+\|Q^n(\vf h_*)\|_{W^{1,1}}\\
&=\|Q^n(\vf h_*)\|_{W^{1,1}}\leq\Const \|\vf\|_{\cC^1} \nu^n
\end{split}
\]
for some $\nu<1$. Thus the quantity in the last line of \eqref{eq:var} is uniformly bounded in $n$ and the quantity in the square bracket of the next to the last line is well defined. Accordingly,
\begin{equation}\label{eq:Vonest}
\left\|\frac 1n\sum_{k=0}^{n-1} \vf\circ f^k(x)-\mu(\vf)\right\|_{L^2(\mu)}^2\leq\Const \frac { \|\vf\|_{\cC^1}}n.
\end{equation}
The above is a refinement, in the special case of expanding maps, of Von Neumann mean ergodic Theorem. Indeed, Von Neumann Theorem, together with the ergodicity of $\mu$, implies that the left hand side of the equation \eqref{eq:Vonest} tends to zero but without any information on the speed of convergence. Since $h_*>0$, it also provides and alternative solution to Problem \ref{prob:mix}. In addition it can be used to prove the almost sure convergence of the ergodic averages.\footnote{ Use the usual trick to study the sum in blocks of size $2^k$.} The latter follows also from the Birkhoff ergodic Theorem since $h_*>0$.

Summarizing: the ergodic average converges Lebesgue almost everywhere to the average with respect to the unique invariant measure absolutely continuous with respect to Lebesgue. A natural question is: what is the exact speed of convergence?
\subsection{The Central Limit Theorem}
\label{subsec:clt}
Let $\vf\in \cC^1(\bT,\bR)$ and set $\hat \vf:=\vf-\mu(\vf)$, then we know that
\[
\lim_{n\to\infty}\frac
1n\sum_{k=0}^{n-1}\hat \vf\circ f^k(x)=0 \quad \Leb-\text{a.e.}
\]
and \eqref{eq:Vonest} suggests that $\frac 1n\sum_{k=0}^{n-1}\hat \vf\circ f^k(x)$ is of size $\cO(n^{-\frac 12})$. It is then tempting to define
\[
\Psi_n:=\frac
1{\sqrt n}\sum_{k=0}^{n-1}\hat \vf\circ f^k.
\]
Accordingly, $\Psi_n$ a random variable with distribution $F_n(t):=\mu(\{x\;:\; \Psi_n(x)\leq t\})$. It
is well know that, for each continuous function $g$ holds\footnote{ If
$g\in\cC^1_0$, then
\[
\int_\bR g dF_n=-\int_\bR F_n(t)g'(t)dt
=-\int_\bR dt\int_{\bT^1}dx\, h_*(x)\Id_{\{z\;:\;\Psi_n(z)\leq t\}}(x)g'(t).
\]
Applying Fubini yields
\[
\begin{split}
\int_\bR g dF_n&=-\int_{\bT^1}dx\int_\bR dt\, h_*(x)\Id_{\{z\;:\;\Psi_n(z)\leq t\}}(x)g'(t)
= -\int_{\bT^1}dx\, h_*(x)\int_{\Psi_n(x)}^\infty g'(t) dt\\
&=\int_{\bT^1}dx\, h_*(x)g(\Psi_n(x)).
\end{split}
\]
The results for $g\in\cC^0$ follows by density.
}
\begin{equation}\label{eq:palle1}
\mu(g(\Psi_n))=\int_{\bR}g(t)dF_n(t)
\end{equation}
where the integral is a Riemann-Stieltjes integral. It is thus clear
that if we can control the distribution $F_n$, we have a very sharp
understanding of the probability to have small deviations (of order
$\sqrt n$) from the limit.

This can be achieved in various ways. In the following, I choose to compute the {\em characteristic function} 
\[
\vf_n(\lambda)=\int_{\bR}e^{i\lambda t}dF_n(t)
\]
of the distribution $F_n$ since this provides the strongest results, but see \cite{Li1} for a softer approach or \cite{Dosp, DeL0} for a more general approach. 

The  characteristic function determines the distribution via the formula
\begin{equation}\label{eq:inv}
F_n(b)-F_n(a)=\lim_{\Lambda\to\infty}\frac
1{2\pi}\int_{-\Lambda}^\Lambda\frac{e^{-ia\lambda}-e^{-ib\lambda}}{i\lambda}
\vf_n(\lambda)d\lambda,
\end{equation}
as can be seen in any basic book of probability theory, e.g. \cite{V1, V2}.
In the case when there exists a density, that is an $L^1$ function $f_n$
such that $F_n(b)-F_n(a)=\int_a^bf_n(t)dt$, then the formula above
becomes simply
\begin{equation}\label{eq:inv1}
f_n(y)=\frac
1{2\pi}\int_\bR e^{-iy\lambda}\vf_n(\lambda)d\lambda,
\end{equation}
and follows trivially from the inversion of the Fourier transform.

Recalling \eqref{eq:palle1}, we can thus start to compute
\begin{equation}\label{eq:char1}
\begin{split}
\vf_n(\lambda)&=\int_{\bT^1}e^{i\lambda \Psi_n(x)}h_*(x)dx\\
&=\int_{\bT^1}e^{i\frac \lambda{\sqrt n}\sum_{k=0}^{n-2}\hat \vf\circ f^k}\circ f(x)\cdot e^{i\frac \lambda{\sqrt n}\vf(x)} h_*(x)dx\\
&=\int_{\bT^1}e^{i\frac \lambda{\sqrt n}\sum_{k=0}^{n-2}\hat \vf\circ f^k(x)}\cdot \cL\left(e^{i\frac \lambda{\sqrt n}\vf} h_*\right)(x)dx.
\end{split}
\end{equation}

It is then natural to define, for each $\nu\in\bR$, the operator
\begin{equation}\label{eq:op-l}
\cL_\nu  h(x)=\left[\cL\left(e^{i\nu\vf}h\right)\right](x).
\end{equation}
This idea is due to Nagaev and Guivarch \cite{Na57,Guivarch}.
Using such an operator we can rewrite \eqref{eq:char1} as
\begin{equation}\label{eq:char2}
\begin{split}
\vf_n(\lambda)&=\int_{\bT^1}e^{i\frac \lambda{\sqrt n}\sum_{k=0}^{n-2}\hat \vf\circ f^k(x)}\cdot \cL_{\frac \lambda{\sqrt n}}\left( h_*\right)(x)dx\\
&=\int_{\bT^1} \cL_{\frac \lambda{\sqrt n}}^n\left( h_*\right)(x)dx,
\end{split}
\end{equation}
where the last line is obtained by iterating the previous arguments.

To conclude we must understand the growth of $\cL_{\frac \lambda{\sqrt n}}^n$. That is, we want to understand the spectrum of the operators $\cL_\nu$ for moderately large $\nu$. Since for $\nu=0$ we know the spectrum we can start by applying perturbation theory.
\begin{lem}\label{lem:pert1} 
There exists $\nu_0, C_0>0$ and  $\xi\in (0,1)$ such that, for all $\nu\in [0,\nu_0]$, we can write $\cL_\nu=\lambda_\nu\Pi_\nu+Q_\nu$ where all the quantities are analytic in $\nu$ and 
\[
\begin{split}
&\Pi_\nu(\vf)=h_\nu\ell_\nu(\vf)\;;\quad\ell_\nu(h_\nu)=1\\
&|\lambda_\nu-1- \frac 12\sigma^2\nu^2|\leq C_0\nu^3\\
&\left\|\Pi_\nu-\Pi_0-\nu\sum_{k=0}^\infty \cL_0^k(\Id-\Pi)\cL'_0\Pi
+\nu\sum_{k=0}^\infty \Pi\cL'_0(\Id-\Pi)\cL_0^k(\Id-\Pi)\right\|_{W^{1,1}}\leq C_0\nu^2\\
&\sigma^2=\int_{\bT}\hat \vf(x)^2h_*(x) dx+2\sum_{k=1}^{\infty}\int_{\bT}\hat \vf\circ f^{k}(x) \cdot \hat \vf(x)\cdot h_*(x) dx\\
&\|Q_\nu^n\|_{W^{1,1}}\leq C_0\xi^n,
\end{split}
\]
where we have used $\;'$ for the derivative with respect to $\nu$ and set $\cL=\cL_0$, $\Pi=\Pi_0$.

In addition, $\sigma=0$ iff there exists $g\in\cC^0(\bT,\bR)$ such that $\hat\vf=g-g\circ f$ (i.e., {\em $\hat\vf$ is a continuous coboundary}).
\end{lem}
\begin{proof}
The spectral decomposition  $\cL_\nu=\lambda_\nu\Pi_\nu+Q_\nu$, its analyticity and the bound on $Q_\nu$ follow by standard perturbation theory, e.g. see \cite{kato}. Moreover, $\Pi_\nu^2=\Pi_\nu$, $\cL_\nu\Pi_\nu=\Pi_\nu\cL_\nu=\lambda_\nu\Pi_\nu$ and $\Pi_\nu Q_\nu=Q_\nu\Pi_\nu=0$. Recall that $\lambda_0=1$ and $\Pi_0=\Leb\otimes h_*$.

Next, we must Taylor expand in $\nu$ the various objects. First of all note that, since the projector $\Pi_0=h_*\otimes \Leb$ is a rank one operator, so is the projector $\Pi_\nu$. Hence, there exists a unique $h_\nu$, $\int_\bT h_\nu(x)dx=1$, in the range of $\Pi_\nu$. Next, chose $\ell_\nu\in (W^{1,1})'$ to have the same kernel as $\Pi_\nu$ and normalise it so that $\ell_\nu(h_\nu)=1$, it follows that $\Pi_\nu(\vf)=h_\nu\ell_\nu(\vf)$. Moreover, 
\[
\cL'_\nu\Pi_\nu+\cL_\nu\Pi'_\nu=\lambda'_\nu\Pi_\nu+\lambda_\nu\Pi_\nu'.
\]
Multiplying by $\Pi_\nu$ from the left, yields
\begin{equation}\label{eq:lambda1}
\lambda'_\nu\Pi_\nu=\Pi_\nu\cL'_\nu\Pi_\nu=\ell_\nu(\cL'_\nu h_\nu)\Pi_\nu
\end{equation}
which, since $\cL'_\nu h=\cL_\nu(i\hat \vf h)$, gives
\[
\lambda'_\nu=i\lambda_\nu\ell_\nu(\hat \vf h_\nu)
\]
and, in particular, $\lambda'_0=0$.

Next, setting $\widehat \cL_\nu=\lambda_\nu^{-1}\cL_\nu$, we have
\[
(\Id-\lambda_\nu^{-1}Q_\nu)(\Id-\Pi_\nu)\Pi_\nu'=(\Id-\widehat\cL_\nu)\Pi_\nu'=\lambda_\nu^{-1}\left[\cL'_\nu\Pi_\nu-\lambda'_\nu\Pi_\nu\right]
=\lambda_\nu^{-1}(\Id-\Pi_\nu)\cL'_\nu\Pi_\nu
\]
which implies 
\begin{equation}\label{eq:pi1}
(\Id-\Pi_\nu)\Pi'_\nu=\lambda_\nu^{-1}\sum_{k=0}^\infty \lambda_\nu^{-k}Q_\nu^k(\Id-\Pi_\nu)\cL'_\nu\Pi_\nu=\lambda_\nu^{-1}\sum_{k=0}^\infty \widehat\cL_\nu^k(\Id-\Pi_\nu)\cL'_\nu\Pi_\nu.
\end{equation}
Note that the above estimates imply that there exists $\nu_0>0$ such that the series is convergent for all $\nu\leq \nu_0$.
Analogously, from $\Pi_\nu\cL_\nu=\lambda_\nu\Pi_\nu$ we obtain
\begin{equation}\label{eq:pi1bis}
\Pi'_\nu(\Id-\Pi_\nu)=\lambda_\nu^{-1}\sum_{k=0}^\infty \Pi_\nu\cL'_\nu(\Id-\Pi_\nu)\widehat\cL_\nu^k(\Id-\Pi_\nu).
\end{equation}
Noticing that $\Pi_\nu'\Pi_\nu+\Pi_\nu\Pi_\nu'=\Pi_\nu'$, that is
\[
\Pi_\nu'\Pi_\nu=(\Id-\Pi_\nu)\Pi_\nu',
\]
implies $\Pi_\nu\Pi_\nu'\Pi_\nu=0$ and $(\Id-\Pi_\nu)\Pi_\nu'(\Id-\Pi_\nu)=0$. We can then write
\begin{equation}\label{eq:pitris}
\begin{split}
\Pi_\nu'&=\Pi_\nu\Pi_\nu'\Pi_\nu+(\Id-\Pi_\nu)\Pi_\nu'\Pi_\nu+\Pi_\nu\Pi_\nu'(\Id-\Pi_\nu)+(\Id-\Pi_\nu)\Pi_\nu'(\Id-\Pi_\nu)\\
&=(\Id-\Pi_\nu)\Pi_\nu'\Pi_\nu+\Pi_\nu\Pi_\nu'(\Id-\Pi_\nu)\\
&=\lambda_\nu^{-1}\sum_{k=0}^\infty \widehat\cL_\nu^k(\Id-\Pi_\nu)\cL'_\nu\Pi_\nu
+\lambda_\nu^{-1}\sum_{k=0}^\infty \Pi_\nu\cL'_\nu(\Id-\Pi_\nu)\widehat\cL_\nu^k(\Id-\Pi_\nu).
\end{split}
\end{equation}
Finally, differentiating \eqref{eq:lambda1}, we have
\[
 \lambda''_\nu\Pi_\nu+ \lambda'_\nu\Pi'_\nu=\Pi_\nu'\cL'_\nu\Pi_\nu+\Pi_\nu\cL''_\nu\Pi_\nu+\Pi_\nu\cL'_\nu\Pi_\nu'
\]
which, multiplying both from left and right by $\Pi_\nu$ yields
\[
\begin{split}
\lambda''_\nu\Pi_\nu&=\Pi_\nu\Pi_\nu'\cL'_\nu\Pi_\nu+\Pi_\nu\cL''_\nu\Pi_\nu+\Pi_\nu\cL'_\nu\Pi_\nu'\Pi_\nu\\
&=\Pi_\nu\Pi_\nu'(\Id-\Pi_\nu)\cL'_\nu\Pi_\nu+\Pi_\nu\cL''_\nu\Pi_\nu+\Pi_\nu\cL'_\nu(\Id-\Pi_\nu)\Pi_\nu'\Pi_\nu.
\end{split}
\]
hence,
\begin{equation}\label{eq:lambda2}
\lambda''_\nu=\ell_\nu\big(\Pi_\nu'(\Id-\Pi_\nu)\cL'_\nu h_\nu+\cL''_\nu h_\nu+\cL'_\nu(\Id-\Pi_\nu)\Pi'_\nu h_\nu\big).
\end{equation}
From the above and equations \eqref{eq:pi1}, \eqref{eq:pi1bis} it follows
\[
\lambda_0''=-\int_{\bT}\hat\vf(x)^2 h_*(x) dx-2\sum_{k=1}^\infty \int_{\bT} \hat\vf\circ f^{k}(x)\hat\vf(x) h_*(x) dx.
\]
Note that \eqref{eq:var} implies that $-\sigma^2=\lambda_0''<0$, thus $\sigma$ is well defined.
We are left with the task of investigating the case $\sigma=0$. Equation \eqref{eq:var} implies that if $\sigma=0$, then $\left\|\sum_{k=0}^{n-1} \hat\vf\circ f^k(x)\right\|_{L^2(\mu)}$ is uniformly bounded in $n$. Accordingly it admits weakly convergent subsequences in $L^2$. Let $g\in L^2$ be an accumulation point, then for each $h\in W^{1,1}$ we have
\[
\begin{split}
\int g \circ f\cdot h \cdot h_*&=\lim_{j\to \infty} \int \sum_{k=0}^{n_j-1} \hat\vf\circ f^k\cdot \cL( h \cdot h_*)=\lim_{j\to \infty} \int \sum_{k=1}^{n_j} \hat\vf\circ f^k\cdot h \cdot h_*\\
&=-\int \hat \vf\cdot  h\cdot  h_*+\lim_{j\to \infty} \int \sum_{k=0}^{n_j-1} \hat\vf\circ f^k\cdot h\cdot  h_*+\int \hat\vf \cL^{n_j}(h\cdot h_*)\\
&=-\int \hat \vf\cdot  h\cdot h_*+\int g \cdot h \cdot h_*.
\end{split}
\]
Since $W^{1,1}$ is dense in $L^2$ it follows
\[
\hat\vf h_*=g h_*-g\circ f h_*,
\]
where, without loss of generality, we can assume $\int g h_*=0$.

It remains to prove that $g\in\cC^0$, this follows from Livsic theory \cite{Livs1, Livs2} but let me provide a simple direct argument: Applying $\cL$ to the last equation yields
\[
\cL\hat\vf h_*=-(\Id-\cL)g h_*.
\]
Since the above equation can be restricted to the space of zero average functions and  $h_*>0$ we can write
\[
g=-\frac{1}{h_*}(\Id-\cL)^{-1}\cL\hat\vf h_*.
\]
The claim follows recalling that $W^{1,1}\subset \cC^0$.
\end{proof}
The next result is not really used in the following but it is necessary to prove the Local Central Limit Theorem and it gives an idea of how the control on larger $\lambda$ allows to obtain sharper results on the limiting distribution.
\begin{lem}\label{lem:pert2}
For each $\nu\neq 0$ we have that the essential spectrum of $\cL_\nu$ acting on $W^{1,1}$ is contained in $\{z\in\cC\;:\; |z|\leq \lambda_*^{-1}\}$ and $\sigma_{W^{1,1}}(\cL_\nu)\subset \{z\in\cC\;:\; |z|<1\}$ provided $\hat\vf$ is not a continuous coboundary.
\end{lem}
\begin{proof}
Since 
\[
\begin{split}
&\|\cL_\nu h\|_{L^1}\leq \|\cL|h|\|_{L^1}\leq \|h\|_{L^1}\\
&\frac d{dx}\cL_\nu h=\cL_{\nu}\left(\frac h{f'}\right)-\cL_\nu\left(\frac{f'' h}{(f')^2}\right)+i\nu\cL(\hat\vf' h)
\end{split}
\]
we have the Lasota-Yorke inequality for the operator $\cL_\nu$. Then Theorem \ref{thm:hennion} implies the inclusion $\sigma_{W^{1,1}}(\cL_\nu)\subset \{z\in\cC\;:\; |z|\leq 1\}$ and that the essential spectral radius is bounded by $\lambda_*^{-1}$. Accordingly the spectral radius can equal one only if it exists $\theta\in\bR$ ad $h\in W^{1,1}$ such that
$\cL_\nu h=e^{i\theta}h$. But then $|h|\leq \cL|h|$ which, integrating yields
\[
0\leq \int \cL|h|(x)-|h|(x) dx=0
\]
that is $\cL|h|=|h|$. Since the eigenvalue one is simple for $\cL$, it must be $h_\nu(x)=e^{i\alpha_\nu(x)}h_*(x)$. As both $h_\nu$ and $h_*>0$ are continuous, it follows that $\alpha_\nu$ can be assumed to be a continuous function without loss of generality. In addition,
\[
\cL h_*(x)=h_*(x)=e^{-i\theta-i\alpha_\nu(x)}\cL_\nu h_\nu(x)=\cL\left(e^{-i\theta-i\alpha_\nu\circ f+i\alpha_\nu+i\nu\hat\vf} h_*\right).
\]
Taking the real part and integrating yields
\[
0=\int_{\bT} \left[1-\cos\left(\theta-\alpha_\nu\circ f(x)+\alpha_\nu(x)+\nu\hat\vf(x)\right)\right]h_*(x) dx
\]
which implies that there exists a function $N:\bT\to \bZ$ such that
\[
\theta-\alpha_\nu\circ f(x)+\alpha_\nu(x)+\nu\hat\vf(x)=2N(x) \pi
\]
Lebesgue almost surely. Hence $N$ must be constant and, taking the average with respect to $\mu$, it follows $2N\pi-\theta=0$. Thus, dividing by $\nu$, we see that $\hat\vf$ is a continuous coboundary.
\end{proof}
Thanks to the above Lemmata we can compute $\vf_n$. For $|\lambda|\leq \nu_0\sqrt n$ we can use Lemma \ref{lem:pert1} and equation \eqref{eq:char2} to write
\begin{equation}\label{eq:clts}
\begin{split}
\vf_n(\lambda)&=e^{-\frac{\sigma^2\lambda^2}{2}+\cO(1/\sqrt n)}+\cO(\xi^n). 
\end{split}
\end{equation}
Next, let $L\geq \nu_0>0$. By Lemma \ref{lem:pert2} we have that the spectral radius of $\cL_{\frac{\lambda}{\sqrt n}}$, for $|\lambda|\in [\nu_0\sqrt n,L\sqrt n]$ is smaller than some $\gamma_L\in (0,1)$.\footnote{ Indeed, the spectral radius is either smaller or equal than $\lambda_*^{-1}$ or it is determined by the point spectrum, and hence varies continuously by standard perturbation theory.} Thus, for $|\lambda|\in [\nu_0\sqrt n,L\sqrt n]$ we have that there exist $C_L>0$ such that
\begin{equation}\label{eq:clts2}
|\vf_n(\lambda)|\leq C_L\gamma_L^n. 
\end{equation}
While it is possible to obtain similar estimates for larger $\lambda$, they are out of the scope of this note (see \cite[Appendix B]{DeL1} for details).
Our estimates do not allow to use \eqref{eq:inv} to compute the distribution $F_n$. This problem can by bypassed in various ways, a simple one is to smooth the density. To this end let  ${\boldsymbol Z}$ be a bounded, independent, zero average random variable so
that $|{\boldsymbol Z}|\le 1$ with smooth density $\psi\in\cC^\infty$.  We can then
consider the random variable $\overline\Psi_{n,\ve}=\Psi_n+\ve {\boldsymbol Z}$ for some $\ve>0$.  The random
variable $\overline\Psi_{n,\ve}$  admits a density, which we denote with $\cN_{n,\ve}$.  In fact, denoting by $\widehat\psi$ the Fourier transform of $\psi$ and using \eqref{eq:inv1}, we have
\[
\begin{split}
\cN_{n,\ve}(y)&=\frac 1{2\pi}\int_{\bR}e^{-i\lambda y}\bE(e^{i\lambda\overline\Psi_n})d\lambda\\
  &=\frac 1{2\pi}\int_{\bR}e^{-i\lambda y}\mu(e^{i\lambda\Psi_n})\widehat\psi(\ve \lambda)d \lambda\\
  &= \frac 1{2\pi}\int_{-\nu_0\sqrt n}^{\nu_0\sqrt n}e^{-i\lambda y}\left[e^{-\frac{\sigma^2\lambda^2}{2}+\cO(1/\sqrt n)}+\cO(\xi^n)\right]\widehat\psi(\ve \lambda)d \lambda\\
  &\phantom{=}+\cO(C_L\gamma_L^n)+\frac 1{2\pi}\int_{|\lambda|\geq L\sqrt n}e^{-i\lambda y}\mu(e^{i\lambda\Psi_n})\widehat\psi(\ve \lambda)d \lambda.
\end{split}
 \]
 To conclude note that, for all $p\in\bN$, $|\widehat\psi(\nu)|\leq C_p \|\psi\|_{\cC^{p+2}}|\nu|^{-p}$ for some $C_p>0$. As an example let us choose $p=4$. Thus, there exists $n_L\in\bN$ such that, for all  $n\geq n_L$,
 \[
\cN_{n,\ve}(y)=\frac{1}{\sigma\sqrt{2\pi}}e^{-\frac{y^2}{2\sigma^2}}+\cO(\frac 1{\sqrt n}+\frac1{\ve^4 L^3 n^{3/2}}).
 \]
 To conclude note that
 \[
 \bP(\overline\Psi_{n,\ve}\in [a+\ve,b-\ve])\leq  \bP(\Psi_{n}\in [a,b])\leq  \bP(\overline\Psi_{n,\ve}\in [a-\ve,b+\ve]).
 \]
 Hence, calling $\bP_{\cG_\sigma}$ the probability distribution of a Gaussian random variable of zero average and variance $\sigma$, we have
 \[
 \begin{split}
 \bP(\Psi_{n}\in [a,b])&\leq \int_{a-\ve}^{b+\ve}\frac{1}{\sigma\sqrt{2\pi}}e^{-\frac{y^2}{2\sigma^2}}dy+|b-a|\cO\left(\frac 1{\sqrt n}+\frac1{\ve^4 L^3 n^{3/2}}\right)\\
&\leq \bP_{\cG_\sigma}([a,b])\left(1+\cO\left(\frac{\ve}{|b-a|}\right)\right)+|b-a|\cO\left(\frac 1{\sqrt n}+\frac1{\ve^4 L^3 n^{3/2}}\right)
\end{split}
 \]
 Arguing similarly for the lower bond and choosing, for example $\ve=n^{-\frac 14}$ and $L=1$ we have, for some $C>0$
 \[
 \left|\bP(\Psi_{n}\in [a,b])-\bP_{\cG_\sigma}([a,b])\right|\leq C\left( n^{-\frac 14}+\frac {|b-a|}{\sqrt n}\right)
 \]
 which gives a non trivial bound for all $a,b\leq \Const \ln n$ and $|b-a|\geq \Const n^{-\frac 14}$ .
The above means that, if the precision of the instrument
is compatible with the statistics, the typical fluctuations in
measurements are of order $\frac 1{\sqrt n}$ and Gaussian. This is well
known by experimentalists who routinely assume that the result of
a measurement is distributed according to a Gaussian.\footnote{Note
however that our proof holds in a very special case that has little to
do with a real experimental setting. To prove the analogous statement
in for a realistic experiment is a completely different ball game.} 
\begin{rem} Note that, if we are not interested in the rate of convergence, then the information that we obtained on the spectral properties of $\cL_\nu$ suffice to prove the Local Limit Theorem.\footnote{ One must use the usual trick to prove the Theorem first for functions with compactly supported Fourier transform and then extend the result by density.}
\end{rem}

\subsection{Perturbation theory}\label{sec:perturb}

Another natural question is: how do the statistical properties of a system depend on small changes in the system? 

Indeed, in real life situations the dynamics is known only with finite precision, hence it is fundamental to know how small changes in the dynamics affects the asymptotic properties of the system.

To answer such a  question we need some type of perturbation
theorem. Several such results are available (e.g., see \cite{Kifer1},
\cite{Vi1} for a review and \cite{BalYo} for some more recent results), here
we will follow mainly the theory developed in \cite{KL} adapted to the
special cases at hand. 

We will start by considering an abstract family of operators $\cL_\ve$
satisfying the following properties.
 
\begin{hyp}\label{cond:pert}
Given two Banach spaces as in Theorem \ref{thm:hennion}, 
consider a family of operators $\cL_\ve\in L(\cB,\cB)$, $\ve\in [0,1]$, with the following properties
\begin{enumerate}
\item Uniform Lasota-Yorke inequality: for all $\ve\in [0,1]$
\[
\|\cL_\ve^n h\|_{\cB}\leq C\lambda^{-n}\|h\|_{\cB}+C\|h\|_{{\cB_w}},\quad
\|\cL_\ve^n h\|_{{\cB_w}}\leq  C\|h\|_{{\cB_w}}\;;
\]
\item $\int \cL_\ve h(x)dx=\int h(x) dx$\;;
\item For $L:\cB\to\cB$ define the norm
\[
|||L|||:=\sup\limits_{\|h\|_{\cB}\leq 1}\|L f\|_{{\cB_w}},
\]
that is the norm of $L$ as an operator from $\cB\to {\cB_w}$. Then there exists $D>0$ such that
\[
|||\cL_0-\cL_\ve|||\leq D\ve.
\]
\end{enumerate}
\end{hyp}

Hypotheses (3) specifies in which sense the family $\cL_\ve$ can be
considered as an approximation of the unperturbed operator $\cL:=\cL_0$. Notice
that the condition is rather weak, in particular the distance between
$\cL_\ve$ and $\cL$ as operators on $\cB$ can be always larger than
$1$. Such a notion of closeness is completely inadequate to apply
standard perturbation theory. To obtain some perturbation results it is
then necessary to restrict the type of perturbations
allowed, this is the content of Hypotheses (1, 2) which state that all the
approximating operators enjoys properties very similar to $\cL$.\footnote{ Actually only Hypotheses (1, 3) are needed in the following.
Hypothesis (2) simply implies that the eigenvalue one is common to all
the operators. If  Hypothesis (2) is not assumed, then the operator $\cL_\ve$ will
always have one eigenvalue close to one, but the spectral radius could
vary slightly, see \cite{LiM} for such a situation.}

To state a precise result consider, for each bounded operator $L$, the set
\[
S_{\delta,r}(L):=\{z\in\bC\;|\; |z|\leq r \hbox{ or dist}(z,\sigma(L))\leq
\delta\}.
\]
Since the complement of $S_{\delta,r}(L)$ belongs to the
resolvent of $L$ it follows that
\[
H_{\delta, r}(L):=\sup\left\{\|(z-L)^{-1}\|_\cB\;|\; z\in\bC\,\backslash
S_{\delta,r}(L)\right\} <\infty.
\]
By $R(z)$ and $R_\ve(z)$ we will mean respectively
$(z-\cL)^{-1}$ and $(z-\cL_\ve)^{-1}$.
\begin{thm}[\cite{KL}]\label{thm:pert}
Consider a family of operators $\cL_\ve:\cB\to\cB$
satisfying Hypotheses (1-3). Let $H_{\delta,r}:=H_{\delta,r}(\cL)$;
$S_{\delta,r}:=S_{\delta,r}(\cL)$, $r>\lambda^{-1}$, $\delta>0$, then there exist $\ve_0,a>0$ such that, for all $\ve\leq \ve_0$,  
$\sigma(\cL_\ve)\subset S_{\delta,r}(\cL)$ and, for each $z\not\in S_{\delta,r}$,
\[
|||R(z)-R_\ve(z)|||\leq C\ve^a.
\]
\end{thm}
A simple, although not optimal, proof can be found in \cite[Theorem 3.2]{Li3}.\footnote{ Formally, the proof in \cite[Theorem 3.2]{Li3} deals with the case $\cB=BV$ and $\cB_w=L^1$, yet it carries out verbatim to the present, more general, case.}
The above perturbation theorem has proven rather flexible and able to cover most of the interesting cases.
\subsection{Deterministic stability}
Let the $\cL_\ve$ be Ruelle-Perron-Frobenius (Transfer) operators of maps $f_\ve$ which 
are $\cC^1$--close to $f$, that is $d_{\cC^1}(f_\ve,\,f)=\ve$ and such
that $d_{\cC^2}(f_\ve,f)\leq M$, for some fixed $M>0$. In this case
the uniform Lasota-Yorke inequality is trivial. On the other hand, for
all $\vf\in\cC^0$ holds 
\[
\int(\cL_\ve h-\cL h)\vf=\int h(\vf\circ f_\ve-\vf\circ f).
\]
Now let $\Phi(x):=(D_xf)^{-1}\int_{f(x)}^{f_\ve(x)}\vf(z)dz$, since
\[
\Phi'(x)=-(D_xf)^{-1}D^2_xf\Phi(x)+ D_x f_\ve(D_xf)^{-1}\vf(f_\ve(x))-\vf(f(x)).
\]
It follows
\[
\int(\cL_\ve h-\cL h)\vf=\int h\Phi'+\int
h(x)[(D_xf)^{-1}D^2_xf\Phi(x)+ (1- D_x f_\ve(D_xf)^{-1})\vf(f_\ve(x))].
\]
Given that $|\Phi|_\infty\leq \lambda^{-1}\ve|\vf|_\infty$ and $|1-
D_x f_\ve(D_xf)^{-1}|_\infty \leq \lambda^{-1}\ve$, we have
\[
\int(\cL_\ve h-\cL h)\vf\leq \|h\|_{\BV}\lambda^{-1}|\vf|_\infty\ve+
|h|_{L^1} \lambda^{-1}(B+1)\ve|\vf|_\infty\leq D\|h\|_{\BV}\ve |\vf|_\infty.
\]
Taking the sup on such $\vf$ yields the wanted inequality
\[
|\cL_\ve h-\cL h|_{L^1}\leq  D\|h\|_{\BV}\ve.
\]
We have thus seen that all the required Hypotheses are satisfied.
See \cite{keller2} for a more general setting including
piecewise smooth maps.

\subsection{Stochastic stability}
Next consider a set of maps $\{f_\omega\}$ depending on a parameter $\omega\in\Omega$. In addition assume that $\Omega$ is a probability space and $P$ a probability measure on $\Omega$.
Consider the process $x_n=f_{\omega_n}\circ \cdots\circ f_{\omega_1}x_0$ where the $\omega$ are i.i.d. random variables distributed accordingly to $P$ and let $\bE$ be the expectation of such process when $x_0$ is distributed according to $\mu$. Then, calling $\cL_{\omega}$ the transfer operator associated to $f_\omega$, we have
\[
\bE(h(x_{n+1})\;|\;x_n)=\cL_Ph(x_n):=\int_\Omega \cL_{\omega}h(x_n) P(d\omega).
\]
If, for all $\omega\in\Omega$, 
\[
|\cL_\omega h|_{\BV}\leq \lambda_{\omega}^{-1}|h|_{\BV}+B_\omega|h|_{L^1},
\]
then integrating yields
\[
|\cL_P h(x)|_{\BV}\leq \bE(\lambda_{\omega}^{-1})|h|_{\BV}+\bE(B_\omega)|h|_{L^1}.
\]
Thus the operator $\cL_P$ satisfies a Lasota-Yorke inequality provided that $\bE(\lambda_\omega^{-1})<1$ and $\bE(B_\omega)<\infty$.

In addition,  if for some map $f$ and associated transfer operator $\cL$,
\[
\bE(|\cL_\omega h-\cL h|)\leq \ve |h|_{\BV}
\]
then we can apply perturbation theory and obtain stochastic stability.

\subsection{Computability}
If we want to compute exactly the invariant measure and the rate of decay of correlations for a specific system we must reduce the problem to a finite dimensional one that can then be solved numerically. To this end we can introduce the function
\[
\phi(x)=\begin{cases}0\quad&\textrm{ if } x<-1\\
x+1 & x\in [-1,0]\\
1-x& x\in [0,1]\\
0&x\geq 1.
\end{cases}
\]
Note that $\sum_{i\in\bZ}\phi(x-i)=1$. We can then introduce the operators
\[
\begin{split}
&P_nh=n\sum_{i=0}^{n-1}\phi(nx-i)\int\phi(ny-i)h(y)dy\\
&\cL_n=P_n\cL.
\end{split}
\]
Note that $P_n(\cC^0)\subset \cC^0$ and
\[
\begin{split}
&\|P_nh\|_{L^1}\leq \|h\|_{L^1}\\
&\|P_nh\|_{W^{1,1}}\leq \|h\|_{W^{1,1}}\\
&\|h-P_nh\|_{L^1}\leq \frac 1n\|h\|_{W^{1,1}}.
\end{split}
\]
So we can again apply Theorem \ref{thm:pert} to show that the finite dimensional operator $\cL_{n}$ has the peripheral spectrum close to the one of $\cL$. The problem is thus reduced to diagonalising a matrix, which can be done numerically (provided the matrix is not too large). There exists a wide literature on the subject, see \cite{Li4} for more details.

\subsection{Linear response}
Linear response is a theory widely used by physicists. In essence it says the follow: consider a one parameter family of systems $f_s$ and the associated (e.g.) invariant measures $\mu_s$, then, for a given observable $\vf$ one want to study the response of the system to a small change in $s$, and, not surprisingly, one expects
$
\mu_s(\vf)=\mu_0(\vf)+s\nu(\vf)+o(s),
$
for some measure or distribution $\nu$.
That is, one expects differentiability in $s$, which is commonly called {\em linear response}. Yet differentiability is not ensured by Theorem \ref{thm:pert}.
It is then natural to ask under which conditions linear response holds. 

For example linear response holds if the maps are sufficiently smooth and the dependence on the parameter is also smooth in an appropriate sense. These type of results follow from a sophistication of Theorem \ref{thm:pert} that can be found in \cite{GL1}.

However, the reader should be aware that there exist natural and relevant cases when linear response fails. See \cite{BS} and references therein for an in depth discussion of this issues.

\section{The contracting case}\label{sec:contractig}
Having illustrated the power of the transfer operator approach in the expanding case, it is natural to investigate to which extent it can be generalised.  A first remark is that, when it works, it automatically implies that the system either does not mix or mixes exponentially fast. Accordingly, the direct application of the above strategy is ill suited to the cases where the decay of correlation is only polynomial (although one  can still apply it after inducing). 

On the contrary, when the decay of correlations is expected to be exponential one can reasonably try to implement a transfer operator approach directly. In particular, it is natural to investigate the possibility to apply it to uniformly hyperbolic systems and partially hyperbolic system. To this end there are several technical  difficulties, some of them still outstanding.   

Clearly the first obstacle is the existence of contracting directions. Hence, our first question is: can we find appropriate Banach spaces for which the transfer operator of a contracting map has good spectral properties? The answer is yes. In fact, again, there exist several  possibilities.\footnote{ They all have the same flavour, although they might be quite different in the details.} 

Let us illustrate a basic one in the simplest possible case: let $f\in\cC^3(\bT,\bT)$ be an orientation preserving diffeomorphim with two fixed points, one attracting and one repelling. Without loss of generality we can assume that zero is the attracting fixed point. Let $\psi\in\cC^2(\bT,\bR)$ be a positive function such that $\psi=1$ in a neighbourhood of zero and $\psi=0$ in a neighbourhood of the repelling fixed point. Also let us assume that the support of $\psi$ be small enough so that 
\[
\|\psi f'\|_{\cC^0}\leq\lambda^{-1}<1.
\]
Consider the transfer operator $\cL h=(\psi h [f']^{-1})\circ f^{-1}$. For a measure $d\mu=h dx$ we have
\[
\int\vf\cL h dx=\int\vf\; d\,[f_*(\psi \mu)].
\]
Hence $\cL$ is the restriction to $L^1$ of the operator $\mu\to f_*(\psi\mu)$. In other words $\cL$ can be naturally extended to the space of measures; abusing notations we will still call $\cL$ such an extension.
With such a notation we have 
\[
\sup_{|\vf|_{\cC^0}\leq 1}\left|\int\vf d(\cL\mu) \right|=\sup_{|\vf|_{\cC^0}\leq 1}\left|\int\vf\circ f\psi  d\mu\right|\leq \sup_{|\vf|_{\cC^0}\leq 1}\left|\int\vf d\mu\right|.
\]
Moreover, $\cL\delta_0=\delta_0$, thus the spectral radius of $\cL$, when acting on the space of measures $\cC^0(\bT,\bR)'$, is one. 
However, as in the previous example, to obtain a Lasota-Yorke inequality we need to consider the operator acting on a different space. This time the space cannot be $\cC^1$ otherwise we would obtain a spectral radius larger than one. We need an idea. 

Idea:  let $\cL$ act on $(\cC^1)'$, the dual of $\cC^1$.\footnote{ The idea is more natural than it may look at first sight: the dual of $\cL$ is, essentially, the composition with $f$, a contractive map. We have seen that, in such a case, looking at the action on $\cC^1$ is a good idea. This suggests to consider $\cL$ acting on the dual of $\cC^1$.} For each $\vf\in\cC^1$, $\|\vf\|_{\cC^1}\leq 1$, we use the following notation\footnote{ This is equivalent to using the same notation for a measure and its density.}
\[
\cL h(\vf)=\int \vf \cL h=\int \vf\circ f \psi  h=h(\vf\circ f \psi),
\]
which is particularly useful when $h\in L^1\subset (\cC^1)'$.
Note that $\|\vf \circ f\psi\|_{\cC^0}\leq \|\vf\|_{\cC^0}$ while $\|(\vf\circ f \psi)'\|_{\cC^0}\leq \lambda^{-1}\|\vf'\|_{\cC^0}+\Const \|\vf\|_{\cC^0}$. The above gives a promising estimate for the derivative but not enough to establish a Lasota-Yorke type inequality. To this end note that, for each $\ve>0$ there exists $\vf_\ve\in\cC^2$ such that $\|\vf_\ve\|_{\cC^1}\leq 1$ and $\|\vf-\vf_\ve||_{\cC^0}\leq \ve$.\footnote{ Simply use a mollifier.} Then, there exists $B_0>0$ such that
\[
\left|\int \vf \cL h\right|\leq \int\left| (\vf-\vf_\ve)\circ f \psi  h\right|+\left|\int \vf_\ve\circ f \psi  h\right|\leq 2\lambda^{-1}\|h\|_{(\cC^1)'}+B_0\|h\|_{(\cC^2)'}
\]
where we have chosen $\ve$ small enough.
\begin{prob}\label{pro:contra1}
Use computations similar to the above to show that there exists $C,B>0$ such that, for all $n\in\bN$ and $h\in(\cC^1)'$,
\begin{equation}\label{eq:LY-contra}
\begin{split}
&\|\cL^n h\|_{(\cC^2)'}\leq C \|h\|_{(\cC^2)'}\\
&\|\cL^n h\|_{(\cC^1)'}\leq C\lambda^{-n}\|h\|_{(\cC^1)'}+B\|h\|_{(\cC^2)'}\,.
\end{split}
\end{equation}
\end{prob}
\begin{prob}\label{pro:contra2}
Prove that the unit ball $\{h\in (\cC^1)'\;:\; \|h\|_{(\cC^1)'}\leq 1\}$ is relatively compact in $(\cC^2)'$.
\end{prob}
Problems \ref{pro:contra1}, \ref{pro:contra2} and Theorem \ref{thm:hennion} imply that $\cL$, when acting on $(\cC^1)'$, has spectral radius one and essential spectral radius bounded by $\lambda^{-1}$. We have already seen that one belongs to the spectra. Suppose that $e^{i\theta}$ is in the spectra, then there exists $h_\theta\in(\cC^1)'$ such that, for all $\vf\in\cC^1$ and $n\in\bN$,
\[
\int e^{i\theta n}h_\theta\vf=\int \cL^n h_\theta \vf=\int h_\theta \left[\prod_{k=0}^{n-1}\psi\circ f^k \right]\vf\circ f^n.
\]
Note that, if $\supp \vf\cap\{0\}=\emptyset$, then there exists $n$ large enough so that $\psi\cdot \vf\circ f^n=0$. By density this implies that $\supp h_\theta=\{0\}$, that is $\int h_\theta \vf=a\vf(0)+b\vf'(0)$. But then we must have, for all $\vf\in\cC^1$,
\[
e^{i\theta}[a\vf(0)+b\vf'(0)]=a\vf(0)+b\vf'(0)f'(0)
\]
which has a solution only for $\theta=0$ and $b=0$. In other words, one is the only eigenvalue of modulus one and it is a simple eigenvalue. It follows that the system is exponentially mixing.\footnote{ To be precise it is exponentially mixing for observables that are supported away from the expanding fixed point. Given the above estimates, it is a simple exercise to study what happens to a general observable.}  Moreover, all the transfer operator theory previously developed can be applied to this situation. Indeed it is a good exercise to do so.

\section{An interlude: Toral automorphisms}\ \newline
The next step is to treat higher dimensional systems in which both contraction and expansion are present. The simplest such case is the {\em uniformly hyperbolic} case in which only expanding and contraction directions are present. Before describing some elements of the general theory we discuss in detail the simplest possible example: Toral automorphisms. For such simple systems we will discuss three different approaches that illustrate the basis of three different general theories used to investigate the statistical properties of dynamical systems.
 
Let us consider the map from $\bT^2$ to itself defined by
\[
f(x)=A x\mod 1, 
\]
with $A\in SL(2,\bZ)$. Also, for simplicity, let us assume that $A^t=A$ and $A_{i,j}>0$.
In analogy with the previous section we can define the operator $\cL h=h\circ f^{-1}$, note that
\[
\int_{\bT^2} \vf \cL h=\int_{\bT^2}\vf\circ f \cdot h.
\]
Simplifying even further, the reader can consider, as a concrete example,
\[
A=\begin{pmatrix}2&1\\1&1\end{pmatrix}.
\]
Note that the Lebesgue measure is invariant since $\det(A)=1$. Moreover $\Tr(A)>2$. Accordingly, the characteristic polynomial reads $t^2+\Tr(A)t+1$ and has roots $\lambda, \lambda^{-1}$, $\lambda>1$. We call $v^u,v^s$ the two normalised vectors such that
\begin{equation}\label{eq:eigen}
\begin{split}
&Av^u=\lambda v^u\\
&Av^s=\lambda^{-1} v^s.
\end{split}
\end{equation}
Note that, since the matrix is assumed symmetric, $\langle v^u,v^s\rangle=0$.

We have thus a natural reference measure. In fact, $(f,\bT^2,\Leb)$ turns out to be {\em mixing}, that is: for each $h,\vf\in\cC^0$
\[
\lim_{n\to\infty}\int_{\bT^2} h(x)\vf(f^n(x))dx=\int_{\bT^2} h(x)dx\int_{\bT^2} \vf(x)dx.
\]
In alternative, the mixing can be stated in the following equivalent way: for each probability measure $\mu$ such that $\frac{d\mu}{d\Leb}=h\in L^1$ and, for each $\vf\in\cC^0$,\footnote{ Recall that $\mu(\vf)=\int_{\bT^2}\vf(x) h(x) dx$ and $f_*\mu(\vf)=\mu(\vf\circ f)$.}
\begin{equation}\label{eq:mixing}
\lim_{n\to\infty}f_*^n\mu(\vf)=\Leb(\vf).
\end{equation}
This is a very relevant property from the applied point of view: it says that asymptotically our system is described by the Lebesgue measure irregardless of the initial distribution (provided the initial condition was distributed according to a measure absolutely continuous with respect to Lebesgue). 

Of course, property \eqref{eq:mixing} is truly useful only if the speed in the convergence to the limit is fast enough. Form this consideration follows the basic question that we want to address in the following:

\noindent{\em What is the speed of convergence in the limit \eqref{eq:mixing}} ?

\subsection{Standard pairs}
The first technique that I am going to illustrate is based on the idea of {\em coupling} in probability. This is a widely used tool to study the convergence to equilibrium of Markov chains. A similar technique was previously used in abstract ergodic theory under the name of {\em joining}. The form I am going to illustrate has been introduced in smooth ergodic theory by Lai-Sang Young \cite{Y2}, further developed in its present form by Dolgopyat and subsequently improved by many people (e.g. \cite{Chernov, Peyman, DeL2}).

The basic idea is to consider a special class of measures that behave under push-forward in a similar way to the case of expanding maps.  Such a class of measures has a long history (e.g. from Pesin and Sinai  \cite{PeSi} to \cite{Li2}), but they have been systematically developed and used by Dolgopyat under the name of {\em standard pairs} \cite{Do2, Do3}.

Fix some $a>1$ and define
\[ 
D_a=\left\{ h\in\cC^0(\bR,\bR_+)\;:\;\forall t,s\in\bR,\; \frac{h(t)}{h(s)}\leq e^{a|t-s|}\right\}.
\]
Also, for each $b\in\bR_+$, $x\in\bT^2$ and $h\in\cC^0(\bR,\bR_+)$, $\int_{-b}^bh=1$, define the measure on $\bT^2$ ({\em standard pair})
\[
\mu_{b,x,h}(\vf)=\int_{-b}^bh(t)\vf(x+tv^u) dt.
\]
The collection of standard pairs will be designated by
\[
S_a=\left\{\mu_{b,x,h}\;:\; b\in[1/2,1],\, x\in\bT^2,\, h\in D_a,\, \int_{-b}^bh=1\right\}.
\]
The above are our building blocks, let us see what we can construct with them. First of all we can take the convex hull: for each finite set $\{p_i\}$ of positive numbers such that $\sum_ip_i=1$  and set $\{\mu_i\}\subset S_a$ we can consider the probability measure
\begin{equation}\label{eq:family}
\mu=\sum_i p_i\mu_i,
\end{equation}
where the $p_i$ are called the {\em masses} of the standard pairs. The set $\{\mu_i,p_i\}$ is called a {\em standard family} and is often confused with the measure it defines via \eqref{eq:family}. Note however that the representation of a measure by a standard family, if it exists, is far from being unique. We will call $\cS_a$ the set of all standard families.
The first important fact is the following.
\begin{lem}\label{lem:lebesgue} The Lebesgue measure belongs to the weak closure of $\cS_a$.\footnote{ Recall that $\mu_n$ converges weakly to $\mu$ if, for all $\vf\in\cC^0$, we have $\lim_{n\to\infty} \mu_n(\vf)=\mu(\vf)$.}
\end{lem}
\begin{proof}
Letting $v^u=(1+u^2)^{-\frac 12}(1, u)$, for each $\vf\in\cC^0$,
\[
\Leb(\vf)=\int_0^1dt\int_0^1 ds\vf(t, s+ut)= \int_0^1 ds\int_0^{\sqrt{1+u^2}}dt\vf(se_2+tv^u).
\]
Note that the the second integral can be written as the convex combination of finitely many standard pairs, the results follows since the first integral is the limit of finite sums.
\end{proof}
Next we want to know how the standard pairs behaves under push forward.
\begin{lem}
For each $n\in\bN$ and $\mu\in \cS_a$ it holds true $f_*^n\mu\in \cS_{\lambda^{-n}a}$.
\end{lem}
\begin{proof}
It suffices to prove that if $\mu\in S_a$ then $f_*^n\mu\in \cS_{\lambda^{-n}a}$. Then, recalling \eqref{eq:eigen},
\[
f_*^n\mu_{b,x,h}(\vf)=\int_{-b}^bh(t)\vf(f^n(x)+t\lambda^nv^u)dt=\lambda^{-n}\int_{-\lambda^{n}b}^{\lambda^{n}b}h(t\lambda^{-n})\vf(f^n(x)+tv^u) dt.
\]
Next, let $\delta\in [1/2,1]$ and $K\in\bN$ such that $\lambda^n b=2K\delta$ and define $t_i=-\lambda^{n}b+(2i+1)\delta$. We can then write
\[
\begin{split}
&f_*^n\mu_{b,x,h}(\vf)=\sum_{i=0}^{K-1}p_i\int_{-\delta}^\delta h_i(t)\vf([f(x)+t_iv^u]+tv^u) dt\\
&p_i=\lambda^{-n}\int_{-\delta}^\delta h(\lambda^{-n}(t_i+t)) dt\\
&h_i(t)=p_i^{-1}h(\lambda^{-n}(t_i+t)).
\end{split}
\]
Accordingly, the Lemma is proven provided $h_i\in D_{\lambda^{-n}a}$. This follows from
\[
\frac{h_i(t)}{ h_i(s)}=\frac{h(\lambda^{-n}(t_i+t))}{h(\lambda^{-n}(t_i+s))}\leq e^{a\lambda^{-n}|t-s|}.
\]
\end{proof}
\begin{rem} Note that the unbounded parameter contraction proven in the previous Lemma is a peculiarity of the linear systems we are studying. However in the nonlinear case a fixed contraction still takes place (provided $a$ is large enough) and this is all we will use in the following.
\end{rem}
To continue, we call two standard pairs $\mu_1=\mu_{b,x,h}$ and $\mu_2=\mu_{b,x+sv^s,h}$, $s\in [1,2]$, {\em matching}, while we call {\em pre-matching} two standard pairs of the form $\mu_1=\mu_{b,x,h_1}$, $\mu_2=\mu_{b,x+sv^s,h_2}$.
The basic fact underlying our strategy is the following:
\begin{lem}\label{lem:matching} Let $\mu_1,\mu_2$ be two matching standard pairs, then, for each $\vf\in\cC^1$,\footnote{ We are using the notation $\partial_s\vf=\langle v^s,\nabla\vf\rangle$.}
\[
\left|f_*^n\mu_1(\vf)-f_*^n\mu_2(\vf)\right|\leq \|\partial_s\vf\|_\infty\lambda^{-n}
\]
\end{lem}
\begin{proof}
It follows by a direct computation:
\[
\begin{split}
\left|f_*^n\mu_1(\vf)-f_*^n\mu_2(\vf)\right|&=\left|\int_{-b}^{b}h(t)[\vf(f^n(x)+\lambda^{-n}sv^s+\lambda^ntv^u)-\vf(f^n(x)+\lambda^ntv^u)]\right|\\
&\leq \|\partial_s\vf\|_\infty\lambda^{-n}\int_{-b}^{b}h(t)=\|\partial_s\vf\|_\infty\lambda^{-n}.
\end{split}
\]
\end{proof}
The above Lemma is really a {\em coupling} between the two measures, see Remark \ref{rem:coupling}. The Lemma shows that the convenient topology in which to study the convergence of the push-forward of standard pairs is $(\cC^1)'$. In other words, it suggests that it is natural to consider {\em distributions} rather than measures. Indeed, this is consistent with our discussion of the contracting case in section \ref{sec:contractig}.

With these definitions in place we are now ready to argue: given two standard pairs $\mu_1,\mu_2$, we know that $f_*^n\mu_1,f_*^n\mu_2$ are standard families in $\cS_{\lambda^{-n}a}$. Note that there is some freedom in how to divide a segment of length $\lambda^{n}b$ in segments of length between $1$ and $2$. In particular one can check that, if $n$ is large enough, one can make the division so that the two families contain two pre-matching standard pairs. That is, there exists a standard pair in the first family supported on $\{y+tv^u\}_{t\in[-b,b]}$ and a standard pair, in the second family, supported on  $\{y+sv^s+tv^u\}_{t\in[-b,b]}$ for some $b\in [1/2,1]$, $s\in [1,2]$ and $y\in\bT^2$. This is a consequence of the fact that the flow $\phi_t(y)=y+tv^u$ is ergodic (although much less is needed), since the ratio of the components of $v^u$ is irrational.

Accordingly, for $n$ large enough, $\lambda^n>2$ and there exist pre-matching standard pairs for any initial couple of standard pairs. Let $n_0$ be the smallest of such $n$. Also we call the two pre-matching standard pairs $\tilde\mu_{0,1}$ and $\tilde\mu_{0,2}$ respectively. Thus we can write\footnote{ We can always arrange so that the two standard families obtained by push forward have the same number of elements $m_1$, for example by allowing some of the $\tilde p_{j,i}$ to be zero or by duplicating the same standard pair giving half of the mass to each copy.}
\[
f_*^{n_0}\mu_1(\vf)-f_*^{n_0}\mu_2(\vf)=\sum_{j=1}^{m_1}\tilde p_{j,1}\tilde \mu_{j,1}(\vf)-\sum_{j=1}^{m_1}\tilde p_{j,2}\tilde \mu_{j,2}(\vf)+\tilde p_{0,1}\tilde\mu_{0,1}(\vf)-\tilde p_{0,2}\tilde\mu_{0,2}(\vf)
\]
for some weights $\tilde p_{j,i}\geq 0$ and standard pairs $\tilde \mu_{j,i}\in \cS_{\lambda^{-n_0}a}$. Note that, if $\tilde p_{j,i}\neq 0$, then $\tilde p_{j,i}\geq (2\lambda^{n_0} e^{2a})^{-1}$ by construction. Also we know that
\[
\tilde\mu_{0,1}(\vf)=\int_{-b_{0}}^{b_{0}} h_{0,1}(t) \vf(y+tv^u) dt\;;\quad \tilde\mu_{0,2}(\vf)=\int_{-b_{0}}^{b_{0}} h_{0,2}(t) \vf(y+sv^s+tv^u) dt
\]
for some $b_{0}\in[1/2,1]$, $y\in\bT^2$ and $h_{0,i}\in D_{\lambda^{-n_0}a}$. 

To obtain a convergence to equilibrium we want to show that some part of the push-forward measures behaves similarly. The tool to do so will be to use Lemma \ref{lem:matching}. To this end we have to exhibit matching standard pairs.

The idea to construct matching standard pairs is to single out a common part of the density by using the fact that $h_{0,i}\geq e^{-2\lambda^{-n_0}a}(2b_0)^{-1}$. Of course we want to still have standard pairs, hence a small computation is called for. For each $c>0$ small enough,
\[
\begin{split}
\frac{h_{0,i}(t)-\frac c{2b_0}}{h_{0,i}(s)-\frac c{2b_0}}&\leq\frac{h_{0,i}(s)e^{\lambda^{-n_0}a|t-s|}-\frac c{2b_0}}{h_{0,i}(s)-\frac c{2b_0}}\leq e^{\lambda^{-n_0}a|t-s|}\frac{h_{0,i}(s)-\frac c{2b_0}e^{-\lambda^{-n_0}a|t-s|}}{h_{0,i}(s)-\frac c{2b_0}}\\
&\leq e^{\lambda^{-n_0}a|t-s|}\left[1+c\frac{1-e^{-\lambda^{-n_0}a|t-s|}}{2e^{-2\lambda^{-n_0}a}-c}\right]
\leq e^{\lambda^{-n_0}a|t-s|}\left[1+c\frac{\lambda^{-n_0}a|t-s|}{2e^{-2\lambda^{-n_0}a}-c}\right].
\end{split}
\]
Finally we choose $c$ so small that 
\[
\gamma=\frac{c}{2e^{-2\lambda^{-n_0}a}-c}\leq 1.
\]
Hence
\[
\frac{h_{0,i}(t)-\frac c{2b_0}}{h_{0,i}(s)-\frac c{2b_0}}\leq e^{\lambda^{-n_0}a(1+\gamma)|t-s|}\leq e^{a|t-s|}.
\]
This means that we can write
\[
\begin{split}
&\tilde\mu_{0,1}(\vf)-\tilde\mu_{0,2}(\vf)=c\int_{-b_{0}}^{b_{0}}\frac 1{2b_0} [ \vf(y+tv^u)-\vf(y+sv^s+tv^u)] dt\\
&+\left(1-c\right)\left[\int_{-b_{0}}^{b_{0}} \frac{h_{0,1}(t)-\frac c{2b_0}}{1-c} \vf(y+tv^u)-\int_{-b_{0}}^{b_{0}} \frac{h_{0,2}(t)-\frac c{2b_0}}{1-c}\vf(y+sv^s+tv^u) dt\right].
\end{split}
\]
Note that we have constructed two matching standard pairs with mass $c$. 

We are almost done. The only remaining problem is that the two pre-matching standard pairs come with different masses. To take care of this we have to rearrange a bit the standard families. Unfortunately the notation is rather unpleasant but if the reader manages to see through the notation she will realise that the strategy is the obvious one.

Let $p_*=\min\{\tilde p_{0,1},\tilde p_{0,2}\}$, $p_{0,i}=\tilde p_{0,i}-p_*c$ and define
\[
\begin{split}
&p_{0,i}=\frac{\tilde p_{0,i}-p_*c}{1-p_*c}\;;\quad p_{j,i}=\frac{\tilde p_{j,i}}{1-p_*c} \quad  \forall j\in \{1,\dots, m_1\}\\
&\mu_{0,1}(\vf)=\int_{-b_{0}}^{b_{0}} \frac {\tilde p_{0,1}h_{0,1}(t) -\frac {p_* c}{2b_0}}{\tilde p_{0,1} -p_*c}\vf(y+tv^u) dt\\
&\mu_{0,2}(\vf)=\int_{-b_{0}}^{b_{0}} \frac {\tilde p_{0,2}h_{0,2}(t) -\frac {p_* c}{2b_0}}{\tilde p_{0,2} -p_*c} \vf(y+sv^s+tv^u) dt\\
&\mu_{0,1}^*(\vf)=\int_{-b_{0}}^{b_{0}}\frac 1{2b_0} \vf(y+tv^u) dt\\
&\mu_{0,2}^*(\vf)=\int_{-b_{0}}^{b_{0}}\frac 1{2b_0} \vf(y+sv^s+tv^u) dt\\
&\mu_{j,i}=\tilde \mu_{j,i} \quad \forall j\in \{1,\dots, m_1\}.
\end{split}
\]
The $\mu_{0,i}^*$ are matching standard pairs, $\mu_{0,i}$ are standard pairs, $\sum_{j=0}^{m_1}p_{j,i}=1$  and
\[
 f_*^{n_0}\mu_i(\vf)=cp_*\mu_{0,i}^*(\vf)+(1-cp_*)\sum_{j=0}^{m_1}p_{j,i}\mu_{j,i}(\vf).
\]
Then, for each $n\geq n_0$,  by Lemma \ref{lem:matching} we have
\[
\begin{split}
&\left| f_*^{n}\mu_1(\vf)-f_*^{n}\mu_2(\vf)-(1-p_*c)\left[\sum_{j=0}^{m_1}p_{j,1}f_*^{n-n_0}\mu_{j,1}(\vf)-\sum_{j=0}^{m_1}p_{j,2}f_*^{n-n_0}\mu_{2,1}(\vf)\right]\right|\\
&\phantom{ f_*^{n}\mu_1(\vf)-f_*^{n}\mu_2(\vf)}
\leq cp_* b\|\partial_s\vf\|_\infty\lambda^{-n+n_0}.
\end{split}
\]
Thus,
\[
\begin{split}
&\left| f_*^{n}\mu_1(\vf)-f_*^{n}\mu_2(\vf)-(1-p_*c)\sum_{j,k=0}^{m_1}p_{j,1}p_{k,2}\left[f_*^{n-n_0}\mu_{j,1}(\vf)-f_*^{n-n_0}\mu_{k,2}(\vf)\right]\right|\\
&\phantom{ f_*^{n}\mu_1(\vf)-f_*^{n}\mu_2(\vf)}
\leq cp_* b\|\partial_s\vf\|_\infty\lambda^{-n+n_0}.
\end{split}
\]
To conclude it suffices to iterate the above formula applying it to each couple of standard pairs $\mu_{j,1}, \mu_{k,2}$. Let $n=\ell n_0$, then for each $\nu<\max\{(1-p_*c)^{1/n_0},\lambda^{-1}\}$ we have
\[
\begin{split}
|f_*^{n}\mu_1(\vf)-f_*^{n}\mu_2(\vf)|&\leq 2(1-p_*c)^{\ell}\|\vf\|_{\infty}+\sum_{k=0}^{\ell-1}cp_* b\|\partial_s\vf\|_\infty(1-p_*c)^{k}\lambda^{-n+(k+1)n_0}\\
&\leq C \nu^n(\|\vf\|_{\infty}+\|\partial_s\vf\|_\infty)
\end{split}
\]
for some $C>0$, depending on $\nu$. The same estimate carries over to standard families and hence to the weak closure of $\cS_a$. The reader can check, arguing similarly to Lemma \ref{lem:lebesgue}, that the above implies that for each $h\in\cC^1$,
\[
\left|\int_{\bT^2}h(x)\vf\circ T^n(x) dx-\int_{\bT^2}\vf(x)dx\right|\leq C(\|h\|_{\infty}+\|\partial_u h\|_\infty)(\|\vf\|_{\infty}+\|\partial_s\vf\|_\infty)\nu^n.
\]
We have thus established that the map is mixing and that the speed of mixing is exponential with a prefactor depending on the smoothness of $h$ along the unstable direction and the smoothness of $\vf$ along the stable direction.

Let us conclude with a general remark connecting the present discussion to usual coupling arguments in probability theory.
\begin{rem}\label{rem:coupling} Given a compact metric space $X$ and two Borel probability measures $\mu,\nu$ a {\em coupling} of the two measures is a probability measure $G$ on $X^2$ such that
\[
\int_{X^2} \vf(x) G(dx,dy)=\int_X \vf(x) \mu(dx) \quad \textrm{and} \quad \int_{X^2} \vf(y) G(dx,dy)=\int_X \vf(y) \nu(dy).
\]
Let $\cG(\mu,\nu)$ be the set of couplings of $\mu$ and $\nu$, we can then introduce the Kantorovich (sometimes called Wasserstein) distance 
\[
d_K(\mu,\nu)=\inf_{G\in\cG(\mu,\nu)} \int _{X^2}d(x,y) G(dx,dy).
\]
The following is a coupling between two matching standard pairs $\mu_1=\mu_{b,x,h}$ and  $\mu_2=\mu_{b,x+sv^s,h}$ :
\[
G(\vf)=\int_{[-b,b]^2} \vf(x+tv^u, x+sv^s+tv^u) h(t) dt.
\]
Using such a coupling we can reinterpret the proof of Lemma \ref{lem:matching} to obtain\footnote{ Indeed, for the stated coupling $G$ of $f_*^n\mu_1,f_*^n\mu_2$,
\[
\left|f_*^n\mu_1(\vf)-f_*^n\mu_2(\vf)\right|=\left| \int_{\bT^4}[\vf(x)-\vf(y) ]G(dx,dy)\right|\leq \|\partial_s\vf\|_\infty d_K(f_*^n\mu_1,f_*^n\mu_2).
\]
}
\[
d_K(f_*^n\mu_1,f^*\mu_2)=\inf_{G'\in\cG(f_*^n\mu_1,f_*^n\mu_2)} \int _{\bT^4} d(x,y) G'(dx,dy)\leq 2be^{ab}\lambda^{-n},
\]
where $d(x,y)=\inf_{k\in\bT^2}\|x-y+k\|$.
Also it is not hard to prove that in this case the topology associated to the distance $d_K$ is the weak topology. As an exercise the reader can translate the results of this section in terms of a statement on the Kantorovich distance.
\end{rem}

\subsection{Fourier Transform}
The standard pairs method is very flexible and can be adapted to a large range of situations. 
Yet, since the maps we are presently studying are linear, a much more powerful too is available: Fourier series. Indeed, for each $k\in\bZ^2$,
\begin{equation}\label{eq:f-dyn}
\begin{split}
(\widehat{\cL^n h})_k&=\int_{\bT^2} e^{2\pi i\langle k, x\rangle}\cL^n h(x) dx=\int_{\bT^2} e^{2\pi i\langle k, A^n x\rangle} h(x) dx\\
&=\int_{\bT^2} e^{2\pi i\langle A^nk, x\rangle} h(x) dx=\hat h_{A^nk}.
\end{split}
\end{equation}
Accordingly, for each $h,\vf\in \cC^r$,
\[
\begin{split}
\left|\int_{\bT^2}\vf\cL^{2n} h-\int \vf\right|&\leq \sum_{k\in\bZ^2/\{0\}}|\hat \vf_k\hat h_{A^{2n}k}|\leq\sum_{k\in\bZ^2/\{0\}}\frac{\|h\|_{\cC^r}\|\vf\|_{\cC^r}}{(\|A^{2n}k\|+1)^r(\|k\|+1)^r}\\
&\leq\sum_{k\in\bZ^2/\{0\}}\frac{\|h\|_{\cC^r}\|\vf\|_{\cC^r}}{(\|A^{n}k\|+1)^r(\|A^{-n}k\|+1)^r} .
\end{split}
\]
For each $k\in\bR^2$, we write $a v^u+bv^s$ (recall \eqref{eq:eigen}). It follows that $A^n k=a\lambda^{n}v^u+b\lambda^{-n}v^s $ and $A^{-n} k=a\lambda^{-n}v^u+b\lambda^{n}v^s $. Thus 
\[
\|A^{-n}k\|^2+\|A^nk\|^2\geq (b^2+a^2)\lambda^{2n}=\|k\|^2\lambda^{2n}.
\]
Accordingly,
\[
(\|A^{n}k\|+1)(\|A^{-n}k\|+1)\geq \|k\|\lambda^n.
\]
We can thus conclude, for all $r>2$,
\[
\left|\int_{\bT^2}\vf\cL^{2n} h-\int \vf\right|\leq\sum_{k\in\bZ^2/\{0\}}\frac{\|h\|_{\cC^r}\|\vf\|_{\cC^r}}{\|k\|^r}\lambda^{-nr}\leq C_r\|h\|_{\cC^r}\|\vf\|_{\cC^r}\lambda^{-nr},
\]
for some constant $C_r$ independent on $h$ and $\vf$.

We have thus proven, again, that toral automorphisms enjoy exponential decay of correlation but we have also uncovered a new phenomena: the speed of decay depends very much on the smoothness of the functions.

Yet, there are also reasons of unhappiness: the requirement on the smoothness of the functions (more than $\cC^2$) is stronger than the one obtained by using standard pairs. In addition our argument does not look very dynamical and seems to take too much advantage of the special features of the example at hand, what to do with a non linear map is highly non obvious.

It would then be very desirable to obtain the above results via a different, more dynamical, strategy. In particular it would be nice if we could find a Banach space on which it is possible to study the spectrum of the operator $\cL$ and such that the above properties can be understood as consequences of the spectral picture.

This can be done in various ways. Let us start with a possibility still based on Fourier transform.

\subsection{A simple class of  {\em Sobolev like} norms}\label{sec:sob-toral}\ \newline
To define a Banach space we can first define a norm on $\cC^\infty(\bT^2,\bC)$ and then we obtain the Banach space by completing $\cC^\infty(\bT^2,\bC)$ with respect to such a norm. 

The usual Sobolev norms are $\|h\|_p^2=\sum_{k\in\bZ^2}\kd^p|\hat h_k|^2$ where $\kd=1+\|k\|^2$ and $p\in\bR$. If $p>0$ then a finite norm implies some regularity while if $p<0$ also distributions can have a finite norm. However we have learned that hyperbolic dynamics have very different behaviour depending on the direction. Typically $\cL^n h$ will be a function regular in the unstable directions but with very wild oscillations in the stable direction. Hence along the stable directions we can have convergence only in a weak sense: in the sense of distributions. To handle this problem different strategy have been proposed, the simplest one is to consider anisotropic Sobolev spaces, that is spaces defined by a norm of the type
\begin{equation}\label{eq:sob}
\|h\|_{p\alpha}^2=\sum_{k\in\bZ^2}\kd^{p\alpha(\hat k)}|\hat h_k|^2
\end{equation}
where $p\in\bR_+$, $\hat k=(k_1:k_2)$ is the projectivization of $k=(k_1,k_2)$, that is the equivalence class containing $k$ with respect to the equivalence relation defined by $k\sim k'$ iff there exists $\lambda \in\bR\setminus\{0\}$ such that $k=\lambda k'$. Finally, $\alpha\in \cC^0(\pP,[-1,1])$. In other words $\alpha$ depends only on the direction of the vector $k$. In  the following, to simplify notations, we will write $\alpha(\hat k)$ as $\alpha (k)$.

We have seen that the action of the dynamics in Fourier coefficients is also given by $A k$. It is then natural to consider the dynamics in the projective space $\pP$. Obviously there are two fixed point $v^u$ and $v^s$ (or, rather, their equivalence classes), the first is attractive while the second is repelling. Fix $\nu\in (\lambda^{-1},1)$, it is easy to check that in $\pP$ there exists intervals $I_+\ni v^u$, $I_-\ni v^s$ and a constant $K>0$ such that\footnote{ Of course, $I_+,I_-$ correspond to cones in the vector space $\bR^2$. I will abuse notation an use $I_+,I_-$ also for the cones of the vectors whose equivalence class belongs to $I_+,I_-$, respectively.}
\[
\begin{split}
&\langle A v \rangle\geq \nu^{-2}\langle v \rangle\quad \textrm{ for all } v\in I_+, \,\|v\|\geq K\\
&\langle A v \rangle\leq \nu^2\langle v \rangle\quad \textrm{ for all } v\in I_-,\, \|v\|\geq K.\\
\end{split}
\]
Let $\hat I_\pm=A^{\pm 1} I_\pm\subset I_\pm$.
We choose then an $\alpha$ with value 1 in $\hat I_+$, value $-1$ in $\hat I_-$ and strictly monotone in between (it is possible to be more explicit about $\alpha$ and optimise it in various ways, but I think it is more important to point out that the above qualitative properties suffice). Note that in $\pP\setminus (\hat I_+\cup \hat I_-)$ we have that $d(v,Av)\geq c$ for some fixed constant $c$,\footnote{ The definition of the distance is not really important, for example the angle between the two vectors will do.} thus there exists $\gamma>0$ such that 
\begin{equation}\label{eq:improve}
\alpha (v)-\alpha(A^{-1}v)\geq \gamma\quad \textrm{ for all } v\not \in I_+\cup  I_-.
\end{equation}
This defines the norm.

From equation \eqref{eq:f-dyn}  it follows that, for all $p\in\bR_+$,
\[
\|\cL h\|_{p\alpha}^2=\sum_{k\in\bZ}\kd^{p\alpha(\hat k)}|\hat h_{Ak}|^2=\sum_{k\in\bZ} \left[\frac{\langle A^{-1}k\rangle^{\alpha(A^{-1} k)}}{\kd^{\alpha(k)}}\right]^p \kd^{p\alpha(k)}|\hat h_{k}|^2.
\]
If $k\in \hat I_+$ and $\|v\|\geq K$ then
\[
\frac{\langle A^{-1}k\rangle^{\alpha(A^{-1} k)}}{\kd^{\alpha(k)}}\leq \frac{\langle A^{-1}k\rangle}{\kd}\leq \nu^2.
\]
If $k\in \hat I_-$ and $\|v\|\geq K$,  then  $Ak\in \hat I_-$ and
\[
\frac{\langle A^{-1}k\rangle^{\alpha(A^{-1} k)}}{\kd^{\alpha(k)}}=\frac{\kd}{\langle A^{-1}k\rangle}\leq \nu^2.
\]
If $k\not\in \hat I_-\cup \hat I_+$ then, setting $B=\|A^{-1}\|$ and recalling \eqref{eq:improve},
\[
\frac{\langle A^{-1}k\rangle^{\alpha(A^{-1} k)}}{\kd^{\alpha(k)}}\leq \frac{\langle A^{-1}k\rangle^{\alpha(k)-\gamma}}{\kd^{\alpha(k)}}\leq B \kd^{-\gamma}.
\]
It is then natural to consider the set\footnote{ Note that $\Gamma$ is a finite set.} 
\[
\Gamma=\{k\in\bZ^2\;:\; \kd\leq \max\{[\nu^{-2} B]^{1/\gamma}, K\}=:L\}.
\]
Hence,
\[
\sup_{k\not\in \Gamma}\frac{\langle A^{-1}k\rangle^{\alpha(A^{-1} k)}}{\kd^{\alpha(k)}}\leq \nu^2,
\]
and the weak norm
\[
\|h\|_w^2=\sum_{k\in \Gamma}|\hat h_k|^2.
\]
We can then write
\begin{equation}\label{eq:ly-fourier}
\|\cL h\|_{p\alpha}\leq \sqrt{\nu^{2p}\|h\|_{p\alpha}^2+B \|h\|_w^2}\leq \nu^p\|h\|_{p\alpha}+B^{2p}L^p \|h\|_w.
\end{equation}
\begin{prob} Use equation \eqref{eq:ly-fourier} to obtain a Lasota-Yoke type inequality for the norms $\|\cdot\|_{p\alpha}, \|\cdot\|_{p\beta}$, $\beta<\alpha$, and deduce the quasi compactness of $\cL$ (recall Remark \ref{rem:compact-ball}).
\end{prob}
For the reader amusement, let us deduce quasi-compactness by an alternative argument. Note that setting $Ph(x)=\sum_{k\in \Gamma}\hat h_k e^{2\pi\langle k,x\rangle}$ we have
\[
\|\cL(1-P)h\|_{p\alpha}\leq  \nu^p\|h\|_{p\alpha}.
\]
We can then set $A=\cL P$ and $Q=\cL(1-P)$, then, for each $\mu>\nu^p$, we can write
\[
(\mu\Id-\cL)=(\Id\mu-Q)^{-1}(\Id-A(\Id\mu-Q)^{-1}).
\]
The claim follows then by the Analytic Fredholm alternative. We then conclude that the essential spectrum of $\cL$ when acting on the Banach space obtained by closing $\cC^\infty$ with respect to the norm $\|\cdot\|_{p\alpha}$ is contained in the set $\{z\in\bC\;:\; |z|\leq \nu^p\}$. To study the discrete spectrum and obtain independently that it consists only of $\{1\}$ requires a little extra argument that we postpone to the end of section \ref{sec:geom-toral}, see Lemma \ref{lem:discrete} if you cannot held your curiosity.

The above it is not as precise as our explicit computation (also due to the choice to reduce the technicalities to a bare minimum) but it provides the main idea of a much far reaching approach.

\subsection{A simple class of {\em geometric} norms}\label{sec:geom-toral}\ \\
We have seen how the anisotropy of the dynamics can be reflected by the norms using a weigh (at time called {\em escape function}) in Fourier transform.
Here we present (always in a simplified manner, adapted to the special case at hand) a different, more geometric, approach that has both advantages (it has been adapted to more general systems, e.g. \cite{BDL}) and disadvantages (for example, the dual of the space is not a space of the same type).
The presentation is a bit more detailed than the one in Section \ref{sec:sob-toral} as we will use it as the base for further generalisations, see Section \ref{sec:hyp}.

Let $\partial_u \vf=\langle v^u, \nabla \vf\rangle$,  fix $\delta >0$, $\vf\in\cC_0^\infty([-\delta,\delta],\bC)$ and $h\in\cC^\infty(\bT^2,\bC)$ define,\footnote{ We use the notation
$\vf^{(q)}(t)=\frac{d^q}{dt^q}\vf (t).$ }
\begin{equation}\label{eq:geo-norm0}
\begin{split}
&|\vf|_q=\sup_{q'\leq q}\sup_{t\in\bR} |\vf^{(q')}(t)|\\
&B_q=\{\vf\in\cC_0^\infty([-\delta,\delta],\bC)\;:\; |\vf|_q\leq 1\}\\
&\|h\|_{p,q}=\sup_{x\in\bT^{2}}\sup_{p'\leq p}\sup_{\vf\in B_q}\int_{-\delta}^{\delta}(\partial^{p'}_uh)(x+tv^s) \cdot \vf(t) dt.
\end{split}
\end{equation}
We will call $\cB^{p,q}$ the closure of $\cC^\infty$ with respect to the above norm. The first thing we want to understand is which kind of objects we obtained by the closure. The next Lemma shows that we are inside the usual space of distributions.
\begin{lem}\label{lem:equivalence}
For each $p,q\in\bN$, $p>0$, we have $i:\cB^{p,q}\to\cC^{q}(\bT^2,\bC)'$, where $i$ is bounded and one-to-one.
\end{lem}
\begin{proof}
As usual, define $i:\cC^\infty(\bT^2,\bC)\to \cC^{q}(\bT^2,\bC)'$ by $i(h)(\vf)=\int_{\bT^2}\vf h$.

Let $\{\phi_i\}_{i=1}^N$ be a smooth partition of unity such that $\supp\phi_i$ is contained in a ball of radius $\delta/2$ with centre $x_i$.
Let $h\in\cC^\infty(\bT^2,\bC)$, for each $\vf\in\cC^q(\bT^2,\bC)$ we have
\[
\begin{split}
\left|i(h)(\vf)\right|&=\left|\int_{\bT^2}h\vf\right|\leq \sum_i\left|\int_{\bT^2}h\vf\phi_i\right|\\
&\leq\sum_i\int_{-\delta}^{\delta}ds \left|\int_{-\delta}^{\delta} dt h(x_i+sv^s+tv^u)(\vf\phi_i)(x_i+sv^s+tv^u)\right|\\
&\leq 2\delta\|h\|_{0,q}\sum_i|\vf\phi_i|_{\cC^q}\leq C_{\delta,q} \|h\|_{p,q} |\vf|_{\cC^q}.
\end{split}
\] 
From which it follows that $i$ is bounded and can be extended to $\cB^{p,q}$. 

Fix $g\in\cC_0^\infty([-1,1],\bR_+)$, $\int g=1$. For each $x\in\bT^2$, $\vf\in\cC_0^\infty([-\delta,\delta],\bC)$ and $\ve>0$ define 
\[
\vf_\ve(y)=\vf(\langle y-x,v^s\rangle )g(\langle y-x,v^u\rangle\ve^{-1})\ve^{-1}.
\]
Then, for $h\in\cC^\infty(\bT^2,\bC)$ we have
\[
\begin{split}
\int h\vf_\ve&=\int ds g(s\ve)\ve^{-1}\int dt h(x+sv^u+tv^s) \vf(t)\\
&=\int dt\, h(x+tv^s) \vf(t)+\cO(\ve\|h\|_{1,q}).
\end{split}
\]
Finally, suppose $i(h)=0$ for some $h\in\cB^{p,q}$. Let $h_n\subset \cC^\infty$ such that $h_n\to h$ in $\cB^{p,q}$, then
\[
\begin{split}
0=i(h)(\vf_\ve)&=\lim_{n\to\infty}\int h_n\vf_\ve\\
&=\lim_{n\to\infty}\int dt\, h_n(x+tv^s)\vf(t)+\cO(\ve\|h_n\|_{1,q})\\
&=\int dt\, h(x+tv^s) \vf(t)+\cO(\ve\|h\|_{1,q}).
\end{split}
\]
Taking the limit $\ve\to 0$ we obtain
\[
0=\int dt\, h(x+tv^s) \vf(t).
\]
Also, since $i(h)(\partial_{u}^{p'}\vf_\ve)=0$, arguing as before and integrating by part yields, for all $p'\leq p$,
\[
0=\int dt\, \partial_u^{p'}h(x+tv^s) \vf(t).
\]
Taking the sup on $x$ we obtain $\|h\|_{p,q}=0$. Hence $i$ is injective.
\end{proof}
Before continuing it is convenient to make sure that the derivative acts in the natural way on the spaces $\cB^{p,q}$.
\begin{lem}\label{lem:embed} For each $p,q\in\bN$ the operator $\partial_u$ is bounded as an operator from $\cB^{p+1,q}$ to $\cB^{p,q}$ and $\partial_s$ is bounded as an operator from $\cB^{p,q}$ to $\cB^{p,q+1}$.
Moreover, their kernels consists of the constants.
\end{lem}
\begin{proof}
The boundedness follows immediately from the definition of the norms (and integration by part in the case of $\partial_s$).  

Next, for each $h\in\cC^\infty$, $x\in\bT^2$ and $\vf\in\cC_0^{q+1}([-\delta,\delta],\bC)$ let us define
\[
h_\vf(x)=\int_{-\delta}^{\delta}h(x+tv^s)\vf(t) dt.
\]
Then
\[
\begin{split}
&\partial_uh_\vf(x)=\int_{-\delta}^{\delta}\partial_u h(x+tv^s)\vf(t) dt\\
&\partial_s h_\vf(x)=\int_{-\delta}^{\delta}\frac d{dt} h(x+tv^s)\vf(t) dt=-\int_{-\delta}^{\delta} h(x+tv^s)\vf'(t) dt.
\end{split}
\]
It follows that $\|\nabla h_\vf\|_\infty\leq \|h\|_{1,q}|\vf|_{q+1}$. Hence, for $h\in\cB^{p,q}$ and $\vf\in\cC^{q+1}$ we have that $h_\vf$ is Lipschitz (it follows by density).

We can now study the equation 
\[
\partial_uh =0
\]
for $h\in\cB^{p+1,q}$. Let $\vf\in\cC^\infty$, then have $h_\vf\in\cC^1$ and $\partial_uh_\vf=0$. This implies $h_\vf=\textrm{const}$.
Accordingly, for each set $Q_{x,\delta}=\{x+sv^s+tv^u\;:\; t,s\in[-\delta,\delta]\}$ and $\vf\in\cC_0^\infty(Q_{x,\delta},\bC)$,
\[
\begin{split}
\int_{\bT^2} h\vf&=\int_{-\delta}^\delta dt\int_{-\delta}^\delta ds h(x+tv^u+sv^s)\vf(x+tv^u+sv^s)\\
&=\int_{-\delta}^\delta dt\int_{-\delta}^\delta ds h(x+sv^s)\vf(x+tv^u+sv^s)
\end{split}
\]
We can then set $\tilde \vf_x(s)=\int_{-\delta}^\delta dt\vf(x+tv^u+sv^s)$ and obtain
\[
\begin{split}
\int_{\bT^2} h\vf&=h_{\tilde\vf_x}(x)=\int_{\bT^2} h_{\tilde\vf_x}(y) dy=\int_{\bT^2}dy\int_{-\delta}^\delta ds\, h(y+sv^s)\tilde\vf_x(s)\\\
&=\int_{\bT^2} h\int_{\bT^2} \vf.
\end{split}
\]
This shows that $h-\int h$ is zero as a distribution, but then, by Lemma \ref{lem:equivalence} it is zero in $\cB^{p+1,q}$, thus the Lemma. Similar arguments holds for the study of the kernel of $\partial_s$.
\end{proof}

\begin{lem}\label{lem:compact} For each $p,q\in\bN$ we have that $\cB^{p+1,q-1}$ embeds compactly in $\cB^{p,q}$.
\end{lem}
\begin{proof}
Since the spaces are separable, it suffices to prove that each sequence $\{h_n\}\subset \cC^\infty(\bT^2,\bC)$, $\|h_n\|_{p+1,q-1}\leq 1$, admits a convergent subsequence.
Using the language of Lemma \ref{lem:embed}, for each $\ve>0$, let $\{x_i\}_{i\in I_\ve}$ be a finite $\ve$ dense set, then for each $h\in\cC^\infty$, $\vf\in B_{q+1}$ there exists $x_i$ such that $\|x-x_i\|\leq \ve$ and
\[
|h_\vf(x)-h_\vf(x_i)|\leq \ve \|\nabla h_\vf\|_\infty\leq \ve \|h\|_{1,q}.
\]
On the other hand, if $|\vf-\tilde\vf|_q\leq \ve$, then
\[
|h_\vf(x_i)-h_{\tilde \vf}(x_i)|\leq \ve \|h\|_{0,q}.
\]
Finally, since the set $B_{q+1}$ is compact in $B_{q}$, there exists a finite set $\{\vf_j\}_{j\in J_\ve}\subset B_{q+1}$ such that, for all $\vf\in B_{q+1}$, $\inf_j|\vf-\vf_j|_q\leq \ve$.
Accordingly,
\[
\|h\|_{p,q+1}\leq \sup_{(i,j)\in I_\ve\times J_\ve} |h_{\vf_j}(x_i)|+\ve\|h\|_{p+1, q}.
\]
We can then conclude by the usual diagonal trick: Note that, for each $\ve>0$, the set $\{(h_n)_{\vf_j}(x_i)\}$ is bounded, thus contained in a compact set, hence it is possible to extract a subsequence $\{h_{n_k}\}$ such that each sequence $(h_{n_k})_{\vf_j}(x_i)$ is converging. Accordingly, we can set $\ve_m=2^{-m}$, and construct recursively the sequences $\{h_{n_{m,k}}\}\subset\{h_{n_{m-1,k}}\}$, $\{h_{n_{0,k}}\}=\{h_k\}$ such that for each $m$ there exists $K_m\in\bN$ such that, for all $k, k'\geq K_m$,
\[
\|h_{n_{m,k}}-h_{n_{m,k'}}\|\leq 2\ve_{m}.
\]
We can then choose the sequence $\tilde h_m=h_{n_{m,K_m}}$, it is easy to check that this is a converging subsequence.
\end{proof}
We have thus described the Banach space, it is now time to study how the transfer operator acts on it.
\begin{lem}[Lasota-Yorke type inequality] For each $h\in\cC^\infty$ and $p,q\in\bN$ we have
\[
\begin{split}
&\|\cL^n h\|_{p,q}\leq \Const \|h\|_{p,q}\\
&\|\cL^n h\|_{p,q}\leq \Const \lambda^{-\min\{p,q\}n}\|h\|_{p,q}+\Const \|h\|_{p-1,q+1}.
\end{split}
\]
\end{lem}
\begin{proof}
Let $h\in\cC^\infty$ and $\vf\in\cC_0^{q}([-a,a],\bC)$, then
\[
\begin{split}
\int_{a}^a(\cL^nh)(x+tv^s) h\vf(t)dt&=\int_{-a}^a h(x+t\lambda^n v^s)\vf(t) dt\\
&=\lambda^{-n}\int_{-\lambda^n a}^{\lambda^n a} h(x+tv^s)\vf(\lambda^{-n} t) dt.
\end{split}
\]
Next, we consider a $\cC^\infty$ partition of unity $\{\phi_i\}$ of $\bR$ such that the elements have support of size $\delta$ and $\|\phi_i\|_{\cC^{q+1}}\leq C$, for some fixed $C>0$. Clearly $[-\lambda^n a,\lambda^n a]$ intersects, at most, $4\lambda^n+1\leq 5\lambda^n$ such elements. Let $t_i$ belong to the support of $\phi_i$. Then
 \begin{equation}\label{eq:first-ly}
\begin{split}
\left|\int_{a}^a(\cL^nh)(x+tv^s) h\vf(t)dt\right|&\leq \sum_i\lambda^{-n}\left|\int_{t_i-\delta}^{t_i+\delta} h(x+tv^s)\vf(\lambda^{-n} t) \phi_i(t)dt\right|\\
&=\sum_i\lambda^{-n}\|h\|_{0,q}\leq 5\|h\|_{0,q}.
\end{split}
\end{equation}
This proves the first inequality of the Lemma for $p=0$. To treat $p>0$ define $\vf_i(t)=\sum_{j=0}^{q-1}\frac{\vf^{j}(\lambda^{-n}t_i)}{j!}\lambda^{-nj}(t-t_i)^j$ and redo the above computation as follows
\[
\begin{split}
\int_{a}^a(\cL^nh)(x+tv^s) h\vf(t)dt&=\sum_i\lambda^{-n}\int_{t_i-\delta}^{t_i+\delta} h(x+tv^s)\vf(\lambda^{-n} t) \phi_i(t)dt\\
&=\sum_i\lambda^{-n}\int_{t_i-\delta}^{t_i+\delta} h(x+tv^s)\left[\vf(\lambda^{-n} t)-\vf_i(t)\right] \phi_i(t)dt\\
&\phantom{=}+\sum_i\lambda^{-n}\int_{t_i-\delta}^{t_i+\delta} h(x+tv^s)\vf_i(t) \phi_i(t)dt.
\end{split}
\]
To continue notice that
\[
\left|\int_{t_i-\delta}^{t_i+\delta} h(x+tv^s)\vf_i(t) \phi_i(t)dt\right|\leq C |\vf|_q\|h\|_{0,q+1},
\]
and 
\[
|\vf(\lambda^{-n} \cdot)-\vf(\lambda^{-n}t_i)|_q\leq C|\vf|_q \lambda^{-nq}.
\]
The above yields
\[
\|\cL^nh\|_{0,q}\leq C\lambda^{-nq}\|h\|_{0,q}+C\|h\|_{0,q+1}.
\]
Next, notice that
\[
\int_{a}^a\partial_u^p(\cL^nh)(x+tv^s) h\vf(t)dt=\lambda^{-np}\int_{a}^a(\cL^n[\partial_i^ph])(x+tv^s) h\vf(t)dt
\]
which, remembering \ref{eq:first-ly}, implies
\[
\|h\|_{p,q}\leq 5\lambda^{np}\|h\|_{p,q}+C\sum_{i=0}^{p-1}\lambda^{n(p-i+q)}\|\partial_u^ih\|_{0,q}+C\|h\|_{p-1,q+1}
\]
which proves the Lemma.
\end{proof}
The above, together with Lemma \ref{lem:compact}, allows to apply Theorem \ref{thm:hennion} and conclude that the essential spectrum of $\cL$, when acting on $\cB^{p,q}$ is bounded by $\lambda^{-p}$.
To complete our alternative derivation of the results obtained by Fourier Transform we need to understand the discrete spectrum. 

\begin{lem}\label{lem:discrete} For each $p,q\in\bN$ we have $\sigma_{\cB^{p,q}}(\cL)\cap\{z\in\bC\;:\; |z|>\lambda^{-p}\}=\{1\}$.
\end{lem}
\begin{proof}
Suppose that $\cL h=\mu h$, $|\mu| >\lambda^{-p}$. Then
\[
\mu\partial_u h=\partial_u\cL h=\lambda^{-1}\cL\partial_u h.
\]
Thus $\partial_u h\in\cB^{p-1,q}$ is an eigenvector of $\cL$ with eigenvalue $\lambda \mu$. Doing it $p$ time we have that $\partial_u^p h\in\cB^{0,q}$ is an eigenvector with eigenvalue $\lambda^p\mu$, but $|\lambda^p\mu|>1$ while the spectral radius of $\cL$ is bounded by one, hence it must be $\partial_u^p h=0$.
But then Lemma \ref{lem:embed}  implies that $\partial_u^{p-1} h$ is constant. Integrating we see that the constant is zero. Iterating this argument $p$ times we have $h=\textrm{const}$, but then $\mu=1$.
\end{proof}

\section{Uniformly hyperbolic maps and Banach spaces}\label{sec:hyp}
In this section we build on what we have learned in the previous sections to treat the general non-linear case in which expanding and contracting directions are both present simultaneously but there is no neutral direction. 

The goal is to develop Banach spaces  on which the  transfer operator has nice properties. This can be done in various way \cite{BKL, GL1, GL2, BT1, BT2, FS}, here we will describe the so called {\em geometrical approach} which generalises the construction detailed in Section \ref{sec:geom-toral}. Alternative approaches are the {\em Sobolev space approach} and the (similar) {\em semiclassical approach}, which generalise the norms detailed in Section \ref{sec:sob-toral}. The description below is intend as an introduction, see \cite{GL1, GL2} for more details and \cite{babook2} for a much more in depth discussion of all the different functional spaces.

In the geometrical approach one would like to divide the stable and unstable direction in such a way that one can integrate along the stable direction, similarly to what we have done in Section \ref{sec:geom-toral}. The simplest possible generalisation would be to integrate on pieces of stable manifold (as in Section \ref{sec:geom-toral}). This is possible (it was indeed the case in the first successful attempts to construct such spaces \cite{BKL}) but it has the draw back that the Banach space depends badly  on the map. Such a feature is very inconvenient if one wants to study an open set of maps, a necessity when investigating  the dependence of the SRB measure from some parameter or in the study of random maps. The construction described  in the following avoids such a problem, at the price of some extra work.

\subsection{Anosov maps}\ \newline
Let us define more precisely the class of maps we want to study: $\cC^r$ Anosov maps, $r\geq 2$.
A diffeomorphism $f\in {\operatorname{Diff}}^r(M,M)$,\footnote{ In fact endomorphisms can be treated in the same way, but let us keep things simple.} where $M$ is a $d$-dimensional compact Riemannian manifold, is called an Anosov map if there exist two uniformly transversal close continuous cones fields $C^u(x), C^s(x)\subset T_xM$ and  $\lambda>1$ such that $d_xf C^u(x)\subset \operatorname{int}C^u(f(x))\cup\{0\}$, $d_xf^{-1} C^s(x)\subset \operatorname{int}C^s(f^{-1}(x))\cup\{0\}$ and
\begin{equation}\label{eq:hyper}
\begin{split}
&\|d_xf v\|> \lambda\|v\|\quad\forall\; v\in C^u(x)\\
&\|d_xf^{-1} v\|> \lambda\|v\|\quad\forall\; v\in C^s(x).
\end{split}
\end{equation}

Note that in higher dimensions cones can have a variety of shapes.\footnote{ A cone is a subset $C$ of a real vector space such that if $v\in C$, then $\lambda v\in C$ for each $\lambda\in\bR$.} We ask that for each $v\in C^u(x)$ there exists a $d^u$ dimensional subspace $E$ of $T_xM$ such that $v\in E\subset C^u(x)$, and for each $v\in C^s(x)$ there exists a $d^s$ dimensional subspace $E$ of $T_xM$ such that $v\in E\subset C^s(x)$.\footnote{ The sophisticated reader will recognise that it might be more elegant to defined the cones as subsets of the Grassmannian.}

It is well known that the above cone invariant and contracting properties are equivalent to the existence of two invariant distributions \cite{KH}. More precisely: at each point $x\in M$ there exists two transversal subspaces $E^s(x)\subset C^s(x)$ and $E^u(x)\subset C^u(x)$ such that $Df E^{u/s}(x)=E^{u/s}(f(x))$ and, in addition, $E^{u/s}(x)$ vary in an H\"older continuous way with respect to $x$.

It is possible to choose an atlas $\{U_i\}_{i=1}^N$ so that for each $U_i$ there exists a special point $x_i\in U_i$, call it the {\em centroid}, such that $D_{x_i}\phi_i E^s(x_i)=\{(\xi,0)\;:\;\xi\in\bR^{d_s}\}$ and $D_{x_i}\phi_i E^u(x_i)=\{(0,\eta)\;:\;\eta\in\bR^{d_u}\}$. Also, without loss of generality, we can assume that $\phi_i(x_i)=0$ and $\phi_i(U_i)=B_{d_s}(0,r_i)\times B_{d_u}(0,r_i)$ where, for all $d'\in\bN$ and $z\in\bR^{d'}$, $B_{d'}(z,r)=\{x\in\bR^{d'}\;:\; \|z-x\|< r\}$.
Clearly, there exists $\delta>0$ such that $M=\cup_i\phi_i^{-1}(B_{d_s}(0,r_i-2\delta)\times B_{d_u}(0,r_i-2\delta))=:\cup_i\widehat U_i$. In other words, a small shrinking  $\{(\widehat U_i, \phi_i)\}_{i=1}^N$ of the charts still forms an atlas. Finally, we can always arrange so that \eqref{eq:hyper} holds with respect to the euclidean norm in the charts for vectors in $\{(0,\eta)\;:\;\eta\in\bR^{d_u}\}$ and $\{(\xi, 0)\;:\;\eta\in\bR^{d_s }\}$, respectively.\footnote{ For example one can use the exponential map at $x_i$ composed with a linear coordinate change  to define the chart.}

By the continuity of the distributions and the contraction of the cones it follows that, provided the $r_i$ are chosen small enough,  the constant cones $C^s_*=\{(\xi,\eta)\in \bR^d\;:\; \|\eta\|\leq \|\xi\|\}$ and $C^u_*=\{(\xi,\eta)\in \bR^d\;:\; \|\xi\|\leq \|\eta\|\}$, are invariant. That is, when the composition makes sense,
\begin{equation}\label{eq:cone-hyp}
\begin{split}
& D\phi_jDfD\phi_i^{-1}C^u_*\subset \operatorname{int}C^u_*\cap\{0\}\\
&D\phi_jDf^{-1}D\phi_i^{-1}C^s_*\subset \operatorname{int}C^s_*\cap\{0\}. 
\end{split}
\end{equation}
\begin{rem} Maps for which there exists cones $C^{u/s}_*$ that satisfy \eqref{eq:cone-hyp} and  the equivalent of   \eqref{eq:hyper}, with respect to the Euclidean norm in the charts, are called {\em cone hyperbolic}. Note that if the map is smooth we just argued that cone hyperbolic is equivalent to Anosov. Yet, the notion of cone hyperbolicity applies more generally, for example to piecewise smooth maps \cite{BG10}.
\end{rem}

\begin{rem}\label{rem:opensets} Note that if $f$ is cone hyperbolic, then there exists a neighbourhood $\cU\subset \cC^1$ such that each $\tilde f\in\cU$ is cone hyperbolic with respect to the same cones.\footnote{ It follows from a standard compactness argument.}
\end{rem}
\subsubsection{Transfer Operator}\ \\
Let us compute the Transfer operator. A change of variable yields\footnote{ Unless differently stated the integrals are always meant with respect to the volume form associated to the metric.}
\[
\int_M h\cdot \vf\circ f=\int_M h\circ f^{-1} |\det Df|^{-1}\circ f^{-1} \vf.
\]
It is then natural to define, for each $h\in\cC^0$, the transfer operator
\begin{equation}\label{eq:TO-gen}
\cL h(x)=(h |\det Df|^{-1})\circ f^{-1}(x).
\end{equation}
The reader can easily check that 
\[
\cL^n h=(h |\det Df^n|^{-1})\circ f^{-n}.
\]
Since
\[
\int_M |\cL h|=\int_M \cL |h| \cdot 1=\int_M |h|\cdot 1\circ f=\int_M |h|,
\]
$\cL$ is a contraction in the $L^1$ norm, hence we would like to define, as in the previous section, a norm for which the spectral radius is one and the essential spectral radius is strictly smaller. In other words, we would like a Banach space on which $\cL$ has spectral radius one and it is quasi-compact.
\subsection{A set of almost stable manifolds}
By the general theory of hyperbolic systems, \cite{KH}, it follows also a less local statement: there exists two invariant 
foliations, the stable and unstable foliations. More precisely, at each point $x\in M$ there exists two local $\cC^r$-manifold $W^s(x)$, $W^u(x)$, of fixed size, such that $W^s(x)\cap W^u(x)=\{x\}$ and, for each  $y\in W^{s/u}(x)$, $E^{s/u}(y)$ is the tangent space to $W^{s/u}(x)$ at $y$.
The invariance means that $f W^u(x)\supset W^u(f(x))$ and $f W^s(x)\subset W^s(f(x))$.

Clearly the above foliations yield a natural candidate for the direction on which to integrate and indeed this was the original approach in \cite{BKL}. However, as already mentioned, such a choice has at least two drawbacks: first, although the manifolds are as regular as the map, the foliation is, in general, only H\"older \cite{KH}. Second, if one would like to have a Banach space in which to analyse not just one map but an open set of maps, then it is necessary to integrate on manifolds that are fairly independent from the map. Both problems have been solved in \cite{GL1}, the idea being to introduce an ``invariant" set of manifolds rather than an invariant distribution (in some sense, the equivalent of an invariant cone, see Remark \ref{rem:opensets}). 

To make precise the above idea it is more transparent to work in charts. Let, $\delta>0$ be small enough and define
\[
\begin{split}
\Sigma^r_i=\bigg\{G\in\cC^r(\bR^{d_s},\bR^{d_u})\;:\;&\|G\|_{\cC^0}\leq  r_i;\;\|DG\|^*_r\leq 1\bigg\},
\end{split}
\]
where $\|\cdot\|^*_r$ is equivalent to the $\|\cdot\|_{\cC^{r-1}}$ norm and will be defined in Lemma \ref{eq:manifolds}.

Given $G\in \Sigma^r_i$ we have $(y,G(y))\in B_{d_s}(0,r_i)\times B_{d_u}(0,r_i)$ for all $y\in  B_{d_s}(0,r_i)$, thus the manifolds 
\begin{equation}\label{eq:manifold}
W_{i,z,G}=\{\phi_i^{-1}(y,G(y))\}_{y\in B_{d_s}(z,\delta)}\;;\quad \widetilde W_{i,z,G}=\{\phi_i^{-1}(y,G(y))\}_{y\in B_{d_s}(z,2\delta)}
\end{equation}
are well defined $d_s$ dimensional $\cC^r$ sub-manifold of $M$ for any $i\in\{1,\dots,N\}$,  $z\in  B_{d_s}(0,r_i-2\delta)$ and $G\in\Sigma^r_i$.
We finally define the announced set of manifolds:
\begin{equation}\label{eq:manifolds}
\Sigma^r=\bigcup_{i=1}^N\; \bigcup_{z\in B_{d_s}(0,r_i-2\delta)} \;\bigcup_{G\in\Sigma^r_i} W_{i, z,G}.
\end{equation}
Given $W=W_{i,z,G}\in\Sigma^r$ we will call $\widetilde W=\widetilde W_{i,z,G}$ its {\em enlargement}.

The above set of manifolds will play the role of the invariant foliation (but it is much more flexible) as is illustrated by the next Lemma.
\begin{lem}\label{lem:invariance}
For each Anosov map $f\in\operatorname{Diff}^r(M)$ there exist norms $\|\cdot\|_{\cC^r}$ and $\|\cdot\|^*_r$, constants $\delta >0$ and $\bar n \in\bN$ such that for all $W\in \Sigma^r$ and $n\geq \bar n$ there exist $m\in\bN$ and a collection $\{W_i\}_{i=1}^m\subset \Sigma^r$ such that,\footnote{ With a bit more work one can prove it for each $\bar n\in\bN$, but let us keep it simple.}
\[
\overline{f^{-n}W}\subset \bigcup_{i=1}^m W_i \subset f^{-n}(\widetilde W).
\]
Moreover, there exists a constant $C_\delta>0$, depending only  on $\delta$, and a partition of unity of $f^{-n}\widetilde W$, subordinated to $\{W_i\}_{i=1}^m\cup\{f^{-n}\widetilde W\setminus \overline{f^{-n}W}\}$, with $\cC^r$ norm bounded by $C_\delta$. That is, a set $\{\pf_i\}_{i=1}^m$ of functions, from $f^{-n}\widetilde W$ to  $[0,1]$, such that $\supp \pf_i\subset W_i$, $\sup_i\|\pf_i\|_{\cC^{r}(W_i,\bR)}\leq C_\delta$, and $\sum_{i=1}^m\pf(x)=1$ for each $x\in \overline{f^{-n}W}$.
\end{lem}
\begin{proof}
Since we will need to control high derivatives it is convenient to use the fact that, for each finite dimensional Banach algebra $\bA$, $\cC^k(\bR^d, \bA)$ is a Banach Algebra as well, provided we choose the right weighted norm. For example 
\begin{equation}\label{eq:b-norms}
\begin{split}
&\|g\|_{\cC^0}=\sup_{x\in\bR^d}\|g(x)\|\\
&\|g\|_{\cC^{k+1}}=\sup_i\|\partial_{x_i} g\|_{\cC^k}+a\| g\|_{\cC^k}
\end{split}
\end{equation}
for $a\geq 2$ will do. Note that this implies\footnote{ Here I use the usual PDE notation in which $\alpha=(i_1,\cdots, i_k)$ is a multiindex, $|\alpha|=k$, and $\partial^\alpha=\partial_{x_{i_1}}\dots\partial_{x_{i_k}}$.} 
\begin{equation}\label{eq:r-norm}
\|g\|_{\cC^{k}}=\sum_{j=0}^k {k\choose j} a^{k-j}\sup_{|\alpha|=j}\|\partial^\alpha g\|_\infty.
\end{equation}
From now on we use such a norm with an $a$ that will be chosen shortly.

Let $W\in \Sigma^r$ and $n\in\bN$ large enough, then $f^{-n} W$ will be  a larger manifold and the distance between the boundaries $\partial f^{-n} W$ and $\partial f^{-n} \widetilde W$ will be (in charts) larger than $2\delta$ due to the backward expansion in the stable cone. First of all note that, for each point $x\in \overline{f^{-n} W}$ there exists $j_x\in\{1,\dots N\}$, $z_x\in B_{d_s}(0,r_{j_x}-2\delta)$ and $G_x\in\cC^r(\bR^{d_s},\bR^{d_u})$, with $x=\phi_{j_x}^{-1}(z_x, G_x(z_x))$ and  $\|G_x\|_\infty\leq r_{j_x}-2\delta$, such that $\widetilde W_{j_x,z_x, G_x}\subset  f^{-n} \widetilde W$. Then $\{W_{j_{x_k},z_{x_k}, G_{x_k}}\}$ covers the closure of a $\delta$ neighbourhood of $f^{-n}W$ in $f^{-n}\widetilde W$.  Accordingly, we can extract a finite covering $\{W_{k}\}_{i=1}^m:=\{W_{j_{x_k},z_{x_k}, G_{x_k}}\}$ of $\overline{f^{-n}W}$ by compactness. The existence of a partition of unity with the wanted properties and subordinated to the covering  is a standards fact, see \cite[Theorem 1.4.10]{Hor}.

To conclude it remains to show that $G_{x_k}\in \Sigma^r_{j_{x_k}}$. Note that, by hypotheses,
\[
D_{f(x)}\phi_jD_xf^{-\bar n}D_{\phi_i(x)}\phi_i^{-1}=\begin{pmatrix} A^{i,j}(x)&B^{i,j}(x)\\C^{i,j}(x)&D^{i,j}(x)\end{pmatrix}
=:\Xi^{i,j}(x)
\]
where, by construction, if $f^{\bar n}(x_i)=x_j$, then 
\begin{equation}\label{eq:model-mat}
\Xi^{i,j}(x_i)=\begin{pmatrix} A^{i,j}_*&0\\0&D^{i,j}_*\end{pmatrix}
\end{equation}
with $\|(A^{i,j}_*)^{-1}\|\leq \lambda^{-\bar n}$ and $\|D^{i,j}_*\|\leq \lambda^{-\bar n}$. Thus, by continuity, for each $\gamma>0$ we can write 
\[
\Xi^{i,j}=\Xi^{i,j}_*+\Delta^{i,j},
\]
where $\Xi^{i,j}_*$ is a constant matrix with the same properties of $\Xi^{i,j}(x_i)$ in \eqref{eq:model-mat} and 
\begin{equation}\label{eq:Delta-n}
\|\Delta^{i,j}\|_\infty\leq \gamma,
\end{equation}
provided the $r_i\geq 2 \delta$ have been chosen small enough.

If $W_{j,\zeta,H}\subset f^{-\bar n}W_{i,z,G}$, then setting $F(x)=\phi_j\circ f^{-\bar n}\circ \phi_i^{-1}$ we have that there exists 
$\alpha\in\cC^r(D,B_{d_s}(0,r_j))$, $D \subset B_{d_s}(0,r_i)$, such that
\begin{equation}\label{eq:contractions}
F(x, G(x))=(\alpha(x), H(\alpha(x))).
\end{equation}
Hence, for each $\xi\in\bR^{d_s}$,
\[
(D\alpha \xi, DH\circ \alpha D\alpha\xi )=\Xi^{i,j}(\xi,DG\xi)=\begin{pmatrix} A^{i,j}&B^{i,j}\\C^{i,j}&D^{i,j} 
\end{pmatrix}\begin{pmatrix} \xi, DG\xi\end{pmatrix}
\]
which implies
\begin{equation}\label{eq:derivatives-ops}
\begin{split}
D\alpha&= A^{i,j}+B^{i,j}DG\\ 
DH&=\left\{(C^{i,j}+D^{i,j}DG)(A^{i,j}+B^{i,j}DG)^{-1}\right\}\circ \alpha^{-1}\\
&=\left\{(C^{i,j}+D^{i,j}DG)(\Id+(A^{i,j})^{-1}B^{i,j}DG)^{-1}(A^{i,j})^{-1}\right\}\circ \alpha^{-1}.
\end{split}
\end{equation}
To estimate the higher order derivatives it is convenient to consider $\Xi^{i,j}$ (and its block constituents) as an operator mapping a vector filed in the chart $i$ to a vector field in the chart $j$. The norm of such an operator is naturally defined to be\footnote{ Note that, by definition, $\|A B\|^*_r\leq\|A\|^*_r\,\|B\|^*_r$.}
\[
\|\Xi\|^*_r=\sup_{\|v\|_{\cC^r}\leq 1}\|\Xi v\|_{\cC^r}.
\]
To estimate such a norm it is helpful the following results.
\begin{sublem}\label{sublem:star}
For each $r\in\bN$ and $\Xi\in\cC^r(\bR^d,GL(\bR^d,\bR^d))$
\begin{equation}\label{eq:matrix-n}
\sup_{|\alpha|\leq r}a^{-|\alpha|}\|\partial^\alpha\Xi\|_\infty\leq \|\Xi\|^*_r\leq e^r (r!)^2 \sup_{|\alpha|\leq r}a^{-|\alpha|}\|\partial^\alpha\Xi\|_\infty.
\end{equation}
\end{sublem}
\begin{proof}
Remembering  \eqref{eq:r-norm} we have
\[
\begin{split}
\|\Xi v\|_{\cC^r}&=\sum_{k=0}^r {{r}\choose{k}} a^{r-k}\sup_{|\alpha|=k}\|\partial^\alpha(\Xi v)\|_\infty
\leq \sum_{k=0}^r {{r}\choose{k}} a^{r-k}\sum_{|\alpha|+|\beta|=k}{k\choose |\beta|}\|\partial^\alpha\Xi\|_\infty\|\partial^\beta v\|_\infty\\
&\leq \sum_{|\beta|=0}^r \sum_{k=|\beta|}^{r}{r\choose |\beta|}\frac{ a^{r-k} r^{k-|\beta|}}{(k-|\beta|)!}\|\partial^{k-|\beta|}\Xi\|_\infty\|\partial^\beta v\|_\infty\\
&\leq \sum_{|\alpha|=0}^{r}\frac{ a^{-|\alpha|} r^{|\alpha|} r!}{|\alpha|!}\|\partial^{\alpha}\Xi\|_\infty\| v\|_{\cC^r}\\
&\leq e^r (r!)^2\sup_{|\alpha|\leq r}a^{-|\alpha|}\|\partial^\alpha\Xi\|_\infty\|v\|_{\cC^r}.
\end{split}
\]
That is
\[
\|\Xi\|^*_r\leq e^r (r!)^2 \sup_{|\alpha|\leq r}a^{-|\alpha|}\|\partial^\alpha\Xi\|_\infty.
\]
On the other hand, if we restrict to $v$ that are constant vector fields with $\|v\|=1$ we have, for each $|\alpha|\leq r$,
\[
\|\Xi\|^*_r\geq a^{-|\alpha|}{r\choose |\alpha|}\sup_{\|v\|=1}\|(\partial^\alpha\Xi) v)\|_\infty\geq a^{-|\alpha|}\|\partial^\alpha \Xi\|_\infty.
\]
\end{proof}
From Sub-Lemma \ref{sublem:star} and equation \eqref{eq:Delta-n} it follows that, by choosing $a$ large enough (depending on $\gamma$ and $\bar n$),
\[
\|\Delta^{i,j}\|^*_r\leq  C_r\gamma.
\]
Accordingly, for each constant $C_{r,d}>1$, choosing $\gamma$ small enough and $\bar n$ large enough, we obtain
\begin{equation}\label{eq:normbound1}
\begin{split}
&\sup_{i,j}\|B^{i,j}\|^*_r+\|C^{i,j}\|^*_r\leq \frac{1}{2C_{r,d}}\\
&\sup_{i,j}\|(A^{i,j})^{-1}\|^*_r\leq  \frac{1}{2C_{r,d}}\\
&\sup_{i,j}\|D^{i,j}\|^*_r\leq  \frac{1}{2C_{r,d}}.
\end{split}
\end{equation}

From the above and equation \eqref{eq:derivatives-ops} it follows
\begin{equation}\label{eq:alpha-der}
\begin{split}
\|(D\alpha)^{-1}\|^*_r&= \|(\Id+(A^{i,j})^{-1}B^{i,j}DG)^{-1} (A^{i,j})^{-1}\|^*_r\\
&\leq \frac 1{2 C_{r,d}}\sum_{k=0}^\infty (\|(A^{i,j})^{-1}B^{i,j}DG)^{-1}\|^*_r)^k\leq \frac 2{3 C_{r,d}}.
\end{split}
\end{equation}
Note that, by similar arguments, we can prove
\begin{equation}\label{eq:alpha-dert}
\|((D\alpha)^t)^{-1}\|^*_r\leq \frac 2{3 C_{r,d}},
\end{equation}
where $A^t$ is the transpose of the matrix $A$.

Unfortunately, to estimate \eqref{eq:derivatives-ops} we need to control the norm of $\Xi\circ \alpha^{-1}$ rather than simply the norm of $\Xi$.
To this end we need another technical Lemma.
\begin{sublem}\label{sublem:alleg}
For each $k\in\bN$ and $\cC^k$ function $g$, we have
\[
\|g\circ \alpha^{-1}\|_{\cC^k}\leq \|g\|_{\cC^k}.
\]
Moreover
\[
\|\Xi\circ\alpha^{-1}\|^r_*\leq C_r\|\Xi\|^r_*.
\]
\end{sublem}
\begin{proof}
By equations \eqref{eq:b-norms} the Lemma is true for $k=0$, moreover we can write
\[
\|g\circ \alpha^{-1}\|_{\cC^{k+1}}=\sup_i\|\partial_{x_i} (g\circ \alpha^{-1})\|_{\cC^k}+a\| g\circ \alpha^{-1}\|_{\cC^k} .
\]
We can thus argue by induction and, remembering \eqref{eq:alpha-dert}, conclude
\[
\begin{split}
\|g\circ \alpha^{-1}\|_{\cC^{k+1}}&\leq\sup_i\|\left[(\partial_{x_j} g) [(D\alpha)^{-1}]_{j,i}\right]\circ \alpha^{-1}\|_{\cC^k}+a\| g\|_{\cC^k}\\
&\leq \| (D\alpha)^t)^{-1}\nabla g\|_{\cC^k}+a\| g\|_{\cC^k}\\
&\leq \| ((D\alpha)^t)^{-1}\|^r_*\|\nabla g\|_{\cC^k}+a\| g\|_{\cC^k}\\
&\leq \frac{2d}{3C_{r,d}} \sup_j\|(\partial_{x_j} g)\|_{\cC^k}+a\| g\|_{\cC^k}\leq \|g\|_{\cC^{k+1}},
\end{split}
\]
provided we have chosen $C_{r,d}$ large enough.

To conclude, recalling \eqref{eq:b-norms}, \eqref{eq:r-norm} and Lemma \ref{sublem:star}
\[
\begin{split}
\|\Xi\circ\alpha^{-1}\|^*_r&\leq e^r (r!)^2 \sup_{|\alpha|\leq r}a^{-|\alpha|}\|\partial^\alpha(\Xi\circ \alpha^{-1})\|_\infty\leq e^r (r!)^2 \sup_{|\alpha|\leq r}a^{-|\alpha|}\|\Xi\circ \alpha^{-1}\|_{\cC^{|\alpha|}}\\
&\leq e^r (r!)^2 \sup_{|\alpha|\leq r}a^{-|\alpha|}\|\Xi\|_{\cC^{|\alpha|}}\leq e^r (r!)^2 \sup_{|\alpha|\leq r}\sum_{j=0}^{|\alpha|} {|\alpha|\choose j} 
\sup_{|\beta|=j}a^{-|\beta|}\|\partial^\beta \Xi\|_\infty\\
&\leq e^r 2^r(r!)^2\|\Xi\|^*_r.
\end{split}
\]

\end{proof}
Applying  Sub-Lemma \ref{sublem:alleg} to formula \eqref{eq:derivatives-ops} and recalling \eqref{eq:normbound1}, \eqref{eq:alpha-der} yields
\[
\begin{split}
\|DH\|^*_r&\leq \|\left\{(C^{i,j}+D^{i,j}DG)(\Id+(A^{i,j})^{-1}B^{i,j}DG)^{-1}(A^{i,j})^{-1}\right\}\circ \alpha^{-1}\|^*_r\\
&\leq C_r\|(C^{i,j}+D^{i,j}DG)(\Id+(A^{i,j})^{-1}B^{i,j}DG)^{-1}(A^{i,j})^{-1}\|^*_r\\
&\leq \frac {2C_r}{6 C_{r,d}^3}(1+\|DG\|^*_r)\leq \frac 23<1,
\end{split}
\] 
provided, again, we have chosen $C_{r,d}$ large enough. This concludes the Lemma.
\end{proof}
\begin{rem} Note that, given $f_0\in\cC^r$ and norms $\|\cdot\|_{\cC^r}, \|\cdot\|_r^*$ for which Lemma \ref{lem:invariance} holds, there exists a neighbourhood $\cU\subset \cC^r$ of  $f_0$ such that Lemma \ref{lem:invariance} holds, with the same norms, for each $f\in\cU$. This is the equivalent of Remark \ref{rem:opensets}.
\end{rem}

\subsection{High regularity norms}\label{sec:norms1}
If $W=W_{i,z,G}\in \Sigma_i^r$ and $\vf\in \cC^k_0(W,\bC)$, we define
\[
|\vf|_{\cC^k}=\|\vf\circ \phi_i^{-1}\circ \bG\|_{\cC^k(B_{d_s}(z,\delta),\bC)}
\]
where, again, $\bG(x)=(x,G(x))$. We are finally ready to define the relevant norms.

For each $p\in \bN$, $q\in\bR_+$ and $h\in\cC^r(M,\bC)$ let\footnote{ Since, by definition, $W$ belongs to one chart we can define $\partial_{x_j}h:=(\partial_{x_j}(h\circ \phi_i^{-1}))\circ \phi_i$.}
  \begin{equation}\label{eq:geo-norm1}
  \|h\|_{p,q} = \sup_{|\alpha|\leq p}\;
  \sup_{W\in\Sigma^r}\;
  b^{|\alpha|}\sup_{ \substack{ \vf \in \cC_0^{q+|\alpha|}(W, \bC)
  \\ |\vf|_{\cC^{q+|\alpha|}} \leq 1}}\; \int_{W}
  [\partial^\alpha h] \cdot \vf,
  \end{equation}
where, for $W=W_{i,z,G}\in \Sigma^r_i$ and $g\in \cC^0(W,\bC)$ we define
\[
\int_W g=\int_{B_{d_s}(z,\delta)} g\circ \phi_i^{-1}(x, G(x)) dx,
\]
and $b$ will be chosen later.
$\cB^{p,q}$ is the closure of $\cC^r(M,\bC)$ with respect to $\|\cdot\|_{p,q}$.

The above norms have been introduced in \cite{GL1} and are the generalisation of the norms \eqref{eq:geo-norm0}. They allow to prove that the transfer operator is quasi compact with essential spectral radius smaller than $\lambda^{\min\{p,q\}}$. 

Here, to simplify the presentation, we discuss only the case $p\leq 1\leq q$ and we do not attempt to obtain sharp bounds. We refer to \cite{GL1} for the general case and more precise estimates.
As done in the previous section we aim at using Hennion's Theorem, to this end we need a Lasota-Yorke type inequality and a compactness result.\footnote{ From now on we consider $\delta$ fixed once an for all, hence we will often not mention the fact that several constants depend on $\delta$.}

\begin{lem}\label{lem:ly-h} 
For each $q\in (0,r-2)$, $p\in\{0,1\}$ and $\nu\in (\lambda^{-\min\{1,q\}},1)$ there exists $C,B>0$ such that, for all $h\in\cC^r(M,\bC)$ and $n\in\bN$,
\[
\begin{split}
&\|\cL^{n} h\|_{0,q}\leq C\|h\|_{0,q}\\
&\|\cL^{n} h\|_{p,q}\leq C\nu^{n}\|h\|_{p,q}+B\|h\|_{0,q+1}.
\end{split}
\]
\end{lem}
\begin{proof}
By a change of variables we have
\[
\int_W \cL^n h \vf=\int_{f^{-n}W} h\, |\det Df^n| J_W f^n\cdot \vf\circ f^n
\]
where $J_Wf^n$ is the Jacobian of the change of variables.\footnote{ Note that we are changing variables on a submanifold, hence the Jacobian differs from $|\det Df^n|$ which corresponds to a change of variables on the full manifold.} We can then use Lemma \ref{lem:invariance} to write
\[
\begin{split}
\left|\int_W \cL^n h \vf\right|&\leq \sum_{j=1}^m\left| \int_{W_j} h\, |\det Df^n|^{-1} J_Wf^n \cdot \vf\circ f^n \pf_j\right|\\
&\leq \|h\|_{0,q}  \sum_{j=1}^m\left||\det Df^n|^{-1} J_Wf^n \cdot \vf\circ f^n \pf_j\right|_{\cC_0^{q}(W_j)},
\end{split}
\]
where $W_j=W_{k_j,z_j,G_j}$.

Remembering Sub-Lemma \ref{sublem:alleg} and equation \eqref{eq:contractions} we can write
\[
\left||\det Df^n|^{-1} J_Wf^n \cdot \vf\circ f^n \pf_j\right|_{\cC_0^{q}(W_j)}\leq C_\delta \left| |\det Df^n|^{-1}\right |_{\cC^{q}(W_j)}
  \cdot \left|J_W f^n\right|_{\cC^{q}(W_j)} \left|\vf\right|_{\cC^{q}_0(W_j)}.
\]
To estimate the above integral we need a technical distortion Lemma.
\begin{sub-lem}[{\cite[Lemma 6.2]{GL1}}]\label{lem:sumdet} There exists $C_\delta>0$ such that, for each $n\in\bN$ and $q\leq r-1$,
holds true
  \[
  \sum_{i=1}^m\left| |\det Df^n|^{-1}\right|_{\cC^{q}(W_j)}
  \cdot \left|J_W f^n\right|_{\cC^{q}(W_j)}\leq C_\delta.
  \]
\end{sub-lem}
\begin{rem} I refer to \cite[Lemma 6.2]{GL1} for the proof, however let me give some intuition about this estimate. If $\lambda_{u}^n, \lambda_s^n$ are, roughly, the expansion and contraction in the unstable and stable directions, respectively, then $ |\det Df^n|^{-1}\sim \lambda_{u}^{-n} \lambda_s^{-n}$ while $J_Wf^n\sim  \lambda_s^{n}$. Hence the summands are roughly equal to $\lambda_{u}^{-n}$. However, if we consider a thickening of size $\lambda_{u}^{-n}$, in the unstable directions, of each $W_i$ then it corresponds to the image of a thickening of size one of $W$ under $f^{-n}$. Since the map is a diffeomorphism, this implies that all such regions are disjoint, thus their total volume (essentially $\sum_j \lambda_{u}^{-n}\delta^{d_s}$) is uniformly bounded by the total volume of $M$, hence the Lemma. The above argument is essentially correct, a part for some standard distortion estimates.
\end{rem}
\noindent Hence we have the first inequality in the statement of the Lemma:\footnote{ Recall that $\delta$ has been fixed and its choice depends only on $f$ and $M$, hence we will no longer keep track of the dependence of the constants from $\delta$. Also we will use, as before, $\Const$ to designate a generic constant depending only on $f$ and $M$.}
\begin{equation} \label{eq:lyuno}
\|\cL^{n} h\|_{0,q}\leq C\|h\|_{0,q}.
\end{equation}
To prove the second inequality we first consider the case $p=0$. We can write\footnote{ E.g., given  a mollifier $j_\ve$ having support $\ve\leq \delta/2$, define $\bar\vf_\ve=\int j_\ve(x-y) \vf\circ \phi_i^{-1}\circ \bG(y) dy$ and $\vf_\ve(z)=\bar\vf_\ve\circ \pi\circ \phi_i(z)$, where $\pi(x_s,x_u)=x_s$.}
\[
\begin{split}
\int_W \cL^n h \vf=\int_{\widetilde W} \cL^n h \vf=\int_{\widetilde W} \cL^n h \vf_\ve+\int_{\widetilde W} \cL^n h (\vf-\vf_\ve).
\end{split}
\]
where $|\vf_\ve-\vf|_{\cC^{q-1}}\leq \ve |\vf|_{\cC^{q}}$,  $|\vf-\vf_\ve|_{\cC^{q}}\leq \Const$ and $|\vf_\ve|_{\cC^{q+1}}\leq \Const \ve^{-1}$. 
It follows\footnote{ We use $\partial_x f^n$ to mean $\partial_x(\phi_i\circ f^n\circ \phi_{k_j}\circ \bG_j)$. Which is nothing else that the contraction of the dynamics in the stable direction.}
\[
\begin{split}
|(\vf- \vf_\ve)\circ f^n|_{\cC^q}&\leq |(\partial^q\vf-\partial^q\vf_\ve)\circ f^n\cdot (\partial_xf^u)^q|_{\cC^0}+\Const |(\vf- \vf_\ve)\circ f^n|_{\cC^{q-1}}\\&\leq\Const \max\{\ve, \lambda^{-qn}\}.
\end{split}
\]
Arguing as before, and choosing $\ve=\lambda^{-qn}$, the above considerations yield
\begin{equation}\label{ly-gen-0}
\| \cL^n h \|_{0,q}\leq \Const \lambda^{-qn}\|h\|_{0,q}+C_n\|h\|_{0,q+1}.
\end{equation}
To continue we must compute
\[
(\partial_{x_k}(\cL^n h\circ \phi_i^{-1}))\circ \phi_{k_j} (x).
\]
To this end we must exchange the order of $\partial_{x_k}$ and $\cL^n$. Unfortunately, doing so will produce a multiplicative factor larger than one due to the contracting directions. A natural idea to overcome this problem is to decompose the vector fields $\partial_{x_k}$ into a vector field along the manifold $W$, that can then be integrated by part without the need of commuting it with $\cL^n$, and a vector field in the unstable direction that, upon exchanging the order of $\partial_{x_k}$ and $\cL^n$ will produce a contracting multiplicative factor. The obstacle to this strategy is that the unstable vector field is, in general, only H\"older, and hence a vector field along the unstable direction cannot have the required regularity. 

To deal with this last problem we will use an approximation  instead of the real unstable direction. Indeed, what is really necessary is that the vector field contracts, while being pushed backward, only for a time $n$. If $E=\{(0,\eta)\in\bR^{d_s}\times \bR^{d_u}\}$, then 
\begin{equation}\label{eq:u-app}
E_n(x)=D_{\phi_i\circ f^{-n}\circ \phi_{k_j}^{-1}(x)}(\phi_{k_j} \circ f^n \circ\phi_i^{-1})E=\{(U_n(x)\eta, \eta)\}_{\eta\in\bR^{d_u}}
\end{equation}
is an $\cC^r$ approximation of the unstable direction with the required property.
\begin{sublem}[{\cite[Appendix A]{GL1}}]\label{sublem:regular-dec}
Given the decomposition \eqref{eq:u-app}, we have 
\[
\|U_n\circ \phi_i\circ f^n\circ \phi_{k_j}^{-1}\circ \bG_j\|_{\cC^r(B_{d_s}(z_j,\delta),\bR^d)}\leq \Const.
\]
\end{sublem}
\begin{rem}
The Lemma is technical and the proof is rather uneventful, so I refer to \cite[Appendix A]{GL1} for the details. However, the reader unwilling to look at another paper can simply carry out a proof by herself using the analogous of \eqref{eq:derivatives-ops} and \eqref{eq:normbound1} in the future rather than the past.
\end{rem}
\begin{sublem}\label{lem:u-decomp}
For each $k\in\{1,\dots,d\}$, $n\in\bN$ and $z\in W\in\Sigma^r$ we can write
\[
e_k= v(z)+w(z)
\]
where $v(z)\in T_zW$, $w(z)\in E_n(\phi_i(z))$ and such that 
\[
|v\circ f^n|_{\cC^r( f^{-n}W,\, \bR^d)}+|w\circ f^n|_{\cC^r(f^{-n}W,\, \bR^d)}\leq \Const.
\]
\end{sublem}
\begin{proof}
Since $T_zW$ and $E_n(\phi_i(z))$ are transversal (the first belong to the stable cone while the second to the unstable one), we can uniquely decompose a vector field along such two subspaces and the decomposed vector field will have uniformly bounded $\cC^0$ norm. It remains only to check is that the decomposition has the required regularity. Since $W$ is a regular manifold, the issue is reduced to analysing $E_n(\phi_i(z))$. The result follows then from Lemma \ref{sublem:regular-dec}. Indeed, the computation boils down to compute the norms of $(\Id-DG U_n)^{-1}\circ \phi_i\circ f^n$ and $(\Id- U_nDG)^{-1}\circ\phi_i\circ  f^n$.  These are uniformly bounded in $\cC^0$, since $\|U_n\|_\infty\|DG\|_\infty<1$ (provided we have chosen the $r_i$ small enough), and the $\cC^k$ norm can be computed by induction recalling the definition \eqref{eq:b-norms}.
\end{proof}
Accordingly, for each $k\in\{1,\dots, d\}$,
\begin{equation}\label{eq:decomp-int-u}
\int_W\vf\partial_{x_k}\cL^n h=\int_W\vf \langle w,\nabla\cL^n h\rangle+\vf \langle v,\nabla\cL^n h\rangle.
\end{equation}
By construction and Lemma \ref{lem:u-decomp} there exists $\tilde w$, $\|\tilde w\|_{\cC^{1+q}(\bR^{d},\bR^{d_s})}\leq \Const$, such that $(\vf w)\circ\phi_i^{-1}\circ\bG=D\phi_i^{-1}D\bG \tilde w$. Hence
\[
\begin{split}
\int_W\langle w,\nabla\rangle\cL^n h&=\int_{B_{d_s}(z,\delta)}\langle D\phi_i^{-1}D\bG \tilde w, \left[ \nabla\cL^n h\right]\circ \phi_i^{-1}(\bG(x)) \rangle dx\\
&=\int_{B_{d_s}(z,\delta)} \langle \tilde w, \nabla\left[(\cL^n h)\circ \phi_i^{-1}\circ\bG\right] \rangle dx\\
&=-\int_{B_{d_s}(z,\delta)} (\operatorname{div}\tilde w) \left[\cL^n h\right]\circ \phi_i^{-1}(\bG(x)) dx\\
&=\int_W \bar \vf \cL^n h
\end{split}
\]
where $\bar\vf= \left[\operatorname{div}\tilde w\right]\circ \pi\circ \phi_i$, $\pi(x,y)=x$. Since $|\bar\vf |_{\cC^{q}}\leq \Const$ by \eqref{eq:lyuno} it follows
\begin{equation}\label{eq:u-decomp-stable}
b\left|\int_W\langle w,\nabla\rangle\cL^n h\right|\leq \Const b \|h\|_{0,q}\leq \Const b \|h\|_{1,q}.
\end{equation}
To conclude we must analyse the second term on the right hand side of equation \eqref{eq:decomp-int-u}. Recalling \eqref{eq:TO-gen} we can write
\[
\begin{split}
\int_W\vf \langle v,\nabla\cL^n h\rangle&=\int_W\vf \langle v,\nabla\left[ (h |\det Df^n|^{-1})\circ f^{-n}\right]\rangle\\
&=\int_W\vf \langle Df^{-n} v,\left[\nabla (h |\det Df^n|^{-1})\right]\circ f^{-n}\rangle\\
&=\int_W\langle \bar v, \cL^n \nabla h\rangle +\int_W \bar \vf \cL^n  h,
\end{split}
\]
where $\bar v=\vf  Df^{-n} v$ and $\bar \vf= \vf \langle Df^{-n} v,\left[\nabla ( |\det Df^n|^{-1})\right]\circ f^{-n}\rangle$. 

By construction we have
$\|\bar v\|_\infty\leq \Const \lambda^{-n}$, and the usual distortion estimated yield $\|\bar v\|_{\cC^{1+q}}\leq \Const \lambda^{-n}$.
We can then use \eqref{eq:lyuno} and the obvious inequality $b\|\partial_{x_j}h\|_{0,q+1}\leq \|h\|_{1,q}$ to write
\begin{equation}\label{eq:u-decomp-unstable}
\begin{split}
b\left|\int_W\vf \langle v,\nabla\cL^n h\rangle\right|&\leq \Const \lambda^{-n}\|h\|_{1,q}+C_n b\|h\|_{q+1}.
\end{split}
\end{equation}
Collecting equations \eqref{ly-gen-0}, \eqref{eq:decomp-int-u}, \eqref{eq:u-decomp-stable} and \eqref{eq:u-decomp-unstable} yields
\[
\|\cL^n h\|_{1,q}\leq C_*\max\{\lambda^{-q},b^{1/n},\lambda^{-1}\}^{n}\|h\|_{1,q}+(b+1)C_n\|h\|_{0,q+1},
\]
for some constant $C_*$.
We are almost done, the only remaining source of unhappiness is that the constant in front of the weak norm seems to depend on $n$, also we have still to choose  $b$. 

Let us first choose the smallest $\bar n$ such that at $C_*\lambda^{-\bar n\min\{q,1\}}\leq \nu^{\bar n}$.  Then we choose 
\[
b=\nu^{\bar n} C_*^{-1}.
\]
At last, for each $n\in\bN$ we write $n=k\bar n+m$, $m<\bar n$, and
\[
\begin{split}
\|\cL^n h\|_{1,q}&\leq \nu^{\bar n}\|\cL^{n-\bar n}h\|_{1,q}+2 C_{\bar n} \|\cL^{n-\bar n}h\|_{0,q+1}\leq \nu^{\bar n}\|\cL^{n-\bar n}h\|_{1,q}+\Const\|h\|_{0,q+1}\\
&\leq\nu^{k\bar n}\|\cL^{m}h\|_{1,q}+\Const\sum_{j=0}^{k-1}\nu^{j\bar n}\|h\|_{0,q+1}\leq \Const \nu^n \|h\|_{1,q}+\Const\|h\|_{0,q+1}.
\end{split}
\]
This concludes the Lemma.
\end{proof}
\begin{rem} Note that the Lasota-Yorke inequality is proven in Lemma \ref{lem:ly-h} only for $h\in\cC^r$. However by density it follows immediately that it holds for all $h\in\cB^{p,q}$.
\end{rem}
The last ingredient of the argument is the compactness of $\cL$.
\begin{lem}\label{lem:compact-h}
For each $q>q'>0$ the operator $\cL: \cB^{1,q'}\to \cB^{0,q}$ is compact.
\end{lem}
\begin{proof}
The proof proceeds along the same lines as Lemma \ref{lem:compact} and is left to the reader as a useful exercise.
\end{proof}
Lemmata \ref{lem:ly-h} and \ref{lem:compact-h}, together with Theorem \ref{thm:hennion}, imply  that $\cL$ has spectral radius one and essential spectral radius bounded by $\nu$. 
\subsection{Low regularity norms}\label{sec:lowreg}
Here we consider norms adapted to maps with minimal regularity. Such norms are inspired to \cite{DL} (of which they constitute a simplification) where they have been developed to treat maps with singularities. Subsequently they have been  modified to study  the statistical properties of billiards in \cite{DZ1, DZ2, DZ3, BDL}. However, such norms turn out to be useful also in treating $\cC^{1+\alpha}$ maps, with $\alpha\in (0,1)$.

The problem in handling the $f\in\cC^{1+\alpha}$, $\alpha\in (0,1)$, comes from the fact that $p\in\bN$, thus the minimal, non trivial, allowed $p$ is $1$ while the arguments of the previous section need, at least, that $p\leq\alpha$. To overcome this limitation one must introduce the equivalent of a H\"older or Sobolev norm in the unstable direction. This can be done in many ways, the one proposed in \cite{DL} being the most geometrical. 

The basic idea is that any distribution $h$ that can be integrated along a stable curve naturally gives rise to a function 
\[
\Psi(h) : \Omega_q=\{(W,\vf)\;:\; W\in \Sigma^{1+\alpha}, \|\vf\|_{\cC_0^q(W,\bC)}\leq 1\}\to \bC
\]
defined as
\[
\Psi(h)(W,\vf):=\int_W h\vf .
\]
Thus it suffices to define a distance on $\Omega_q$ and impose and H\"older regularity on $\Psi(h)$ with respect to such a distance.
As we find convenient to work in charts we will define a distance in each $\Omega_{i,q}=\{(W,\vf)\;:\; W\in \Sigma_i^{1+\alpha}, \|\vf\|_{\cC_0^q(W,\bC)}\leq 1\}$. Note that the sets  $\Omega_{i,q}$ are not disjoint, yet we will consider their disjoint union, so an object with two different representations will be treated as two different objects.
Then, for each $(W_{i,z,G}, \vf), (W_{i,z',G'}, \vf')\in \Omega_{i,q}$ we define
\begin{equation}\label{eq:distance}
\begin{split}
d((W_{i,z,G}, \vf),(&W_{i,z',G'}, \vf'))=\|z-z'\|+\|G\circ\tau_z- G'\circ\tau_{z'}\|_{\cC^0(B_{d_s}(0, 2\delta))}\\
&+\|\vf\circ\phi_i^{-1}\circ \bG\circ \tau_{z}-\vf'\circ\phi_i^{-1}\circ \bG'\circ\tau_{z'}\|_{\cC_0^q(B_{d_s}(0, \delta))}
\end{split}
\end{equation}
where $\tau_z(x)=x+z$ and $\bG(x)=(x,G(x))$. The reader can easily check that the above is a semi-metric in $\Omega_{i,q}$. Indeed, two curves with the same centre that differ only outside a ball of radius $2\delta$ have zero distance. This is reasonable as the value of $G$ outside such a ball is totally irrelevant and we defined $G$ on all the space just for convenience, while the introduction of enlarged manifolds was simply a device to avoid invoking some fancy extension theorem to enlarge our manifolds when needed. Thus, it is natural to  consider the equivalence classes with respect to the equivalence relation $W\sim W'$ iff $d(W,W')=0$. In the following we will do so without further mention. We have thus defined a metric and we can now define, for each $p<q<\alpha$, and $a>0$, to be chosen later,
\[
\|h\|_{p,q}=a\|h\|_{0,q-p}+\sup_{i}\sup_{\substack{ (W,\vf),(W',\vf')\in\Omega_i^{q}\\d((W,\vf),(W',\vf')\leq \delta/4}} \frac{\left|\int_W h\vf-\int_{W'}h \vf'\right|}{d((W,\vf),(W',\vf'))^p}.
\]
Once the norms are defined we can again close the $\cC^{1+|\alpha|}$ functions with respect to the norms $\|\cdot\|_{0,q}$ and $\|\cdot\|_{p,q}$ to obtain the spaces $\cB^{0,q}$ and $\cB^{p,q}$, respectively.
Next, we need to prove the Lasota-Yorke inequalities. 
\begin{lem}\label{lem:ly-h1} For each $1>\alpha> q>p>0$ and $\nu\in (\lambda^{-\min\{p,q-p\}},1)$ there exist $C,B>0$ such that, for all $h\in\cC^{1+\alpha}(M,\bC)$,
\[
\begin{split}
&\|\cL^{n} h\|_{0,q}\leq C\|h\|_{0,q}\\
&\|\cL^{n} h\|_{p,q}\leq C\nu^{n}\|h\|_{p,q}+B\|h\|_{0,q}.
\end{split}
\]
\end{lem}
\begin{proof}
The first inequality has been proven in Lemma \ref{lem:ly-h}. In addition, by \eqref{ly-gen-0},\footnote{ Since $\|h\|_{0,q'}\leq \|h\|_{0,q''}$ for all $q''\leq q'$ and $q-p+1>q$.}
\begin{equation}\label{ly-gen-1}
\| \cL^n h \|_{0,q-p}\leq \Const \lambda^{-(q-p)n}\|h\|_{0,q-p}+C_n\|h\|_{0,q}.
\end{equation}

 For the second, let $(W,\vf)=(W_{i,z,G},\vf),(W',\vf')=(W_{i,z',G'},\vf')\in\Omega_i^{q}$ and recall from the beginning of the proof of Lemma \ref{lem:ly-h} that 
\[
\begin{split}
\int_{W_{i,z,G}}\hskip-6pt \cL^n h\vf&=\sum_{j=1}^m\int_{W_{k_j, z_j,G_j}} h|\det Df^n|^{-1} J_Wf^n \cdot \vf\circ f^n \pf_j.
\end{split}
\]
Let $\widehat W_{k_j, z_j,G_j}=\phi_{k_j}^{-1}(\{\bG_j(x)\}_{x\in B_{d_s}(z_j,\delta/2)})$ be the restriction of  $W_{k_j, z_j,G_j}$.
Since the construction of the decomposition holds for any choice of $\delta$, we can arrange so that $\supp \pf_j\subset \widehat W_{k_j, z_j,G_j}$ and that $\cup_j\widehat W_{k_j, z_j,G_j}\supset f^{-n}\widetilde W$. Let $G_{j}'$ be the function describing the part of the graph of $f^{-n}W'$ in the chart $U_{k_j}$ which is $\Const d(W,W')\lambda^{-n}$ close to $W_{k_j,z_j,G_j}$. Then $\{W_{k_j, z_j, G_j'}\}$ is a covering of $f^{-n}W'$. Next we define $\psi_j: W_{k_j, z_j, G_j'}\to W_{k_j, z_j, G_j}$ as
\[
\psi_j(\zeta)=\phi_{k_j}^{-1}\circ \bG_j\circ \pi\circ \phi_{k_j}(\zeta),
\]
where $\pi(x,y)=x$.
Setting $\pf_j'=\pf_j\circ\psi_j$ we have
\[
\pf'_j\circ \phi_{k_j}^{-1}\circ \bG'_j(x)=\pf_j\circ \phi_{k_j}^{-1}\circ \bG_j(x).
\]
If $I_\zeta=\{j\;:\; \pf_j\circ f^{-n}(\zeta)>0\}$, then, by definition, $\sum_{j\in I(\zeta)}\pf_j\circ f^{-n}(\zeta)=1$.   For all $j,j'\in I(\zeta)$, we have $d(\psi_j(\zeta), \psi_{j'}(\zeta))\leq \Const \lambda^{-n}d(W,W')$.\footnote{ Indeed,  $\psi_{j}(\zeta)\neq \psi_{j'}(\zeta)$ only if $k_j\neq k_{j'}$. In such a case the vertical movement in the chart $k_{j'}$ will correspond to a movement in a different vertical direction in the chart $k_j$ (but always inside the unstable cone). Since the  manifolds $W_{k_j,z_j,G_j}$ and $W_{k_j,z_j,G'_j}$ are at a distance less than $\Const \lambda^{-n}d(W,W')$, it follows that the point can move horizontally by at most $\Const \lambda^{-n}d(W,W')$.} 
Accordingly, 
\begin{equation}\label{eq:almost-p}
\left|\sum_j\pf'_j-1\right|_{\cC^1}\leq \Const d(W,W').
\end{equation}
Next we set 
\[
\begin{split}
&Z_j=|\;|\det Df^n|^{-1} J_Wf^n\;|_{\cC^q(W)}\;;\quad
Z'_j=|\;|\det Df^n|^{-1} J_{W'}f^n\;|_{\cC^q(W')}\\
&\gamma_j=Z_j^{-1}|\det Df^n|^{-1} J_Wf^n \;;\quad
\gamma'_j=(Z_j')^{-1}|\det Df^n|^{-1} J_{W'}f^n \\
&\vf_j=\vf\circ f^n\;;\quad \vf'_j=\vf'\circ f^n\\
&\bar \vf_j=\vf\circ f^n\circ \psi_j.
\end{split}
\]
By the usual distortion arguments if follows that
\begin{equation}\label{eq:jacobian}
|Z'_j\gamma'_j- Z_j\gamma_j\circ \psi_j|_{\cC^{\alpha-p}}\leq \Const d(W,W')^p Z_j.
\end{equation}
In addition,
\[
\vf'_j\circ \phi^{-1}_{k_j}\circ\bG'_j(x)-\bar\vf_j\circ \phi_{k_j}^{-1}\circ\bG'_j(x)=\vf'_j\circ \phi_{k_j}^{-1}\circ\bG'_j(x)-\vf_j\circ \phi_{k_j}^{-1}\circ\bG_j(x)
\]
hence, recalling Sub-Lemma \ref{sublem:alleg} and definition \eqref{eq:distance},
\begin{equation}\label{eq:vfbarvf}
|\vf'_j-\bar\vf_j|_{\cC^{q-p}}\leq \Const d((W, \vf),(W',\vf'))^p.
\end{equation}
Then, recalling \eqref{eq:almost-p} and Sub-Lemma \ref{lem:sumdet},
\[
\left|\int_{W_{i,z',G'}}\hskip-12pt \cL^n h\vf'-\sum_{j=1}^m\int_{W_{k_j, z_j,G'_j}}\hskip-.8cm h|\det Df^n|^{-1} J_Wf^n \cdot \vf'\circ f^n \pf_j'\right|\leq\Const \|h\|_{0,q-p}d(W,W').
\]
Moreover, by \eqref{eq:jacobian} and \eqref{eq:vfbarvf},
\[
\left|\int_{W_{i,z',G'}}\hskip-12pt \cL^n h\vf'-\sum_{j=1}^m Z_j\int_{W_{k_j, z_j,G'_j}}\hskip-.8cm h\gamma_j\circ \psi_j \cdot \bar\vf_j \pf_j'\right|\leq\Const \|h\|_{0,q-p}d((W,\vf),(W',\vf'))^p.
\]
We can finally compute 
\[
\begin{split}
\left|\int_{W_{i,z,G}}\hskip-6pt \cL^n h\vf-\int_{W_{i,z',G'}}\hskip-6pt \cL^n h\vf'\right|\leq &\sum_{j=1}^m Z_j\left|\int_{W_{k_j,z_j,G_j}}\hskip-12pt h\gamma_j\vf_j\pf_j
-\int_{W_{k_j,z_j,G'_j}}\hskip-12pt h\gamma_j\circ\psi_j\bar \vf_j\pf'_j\right|\\
&+\Const \|h\|_{0,q-p}d((W,\vf),(W',\vf'))^p.
\end{split}
\]
At last notice that, recalling \eqref{eq:distance},
\[
d((W_{k_j,z_j,G_j}, \gamma_j\vf_j\pf_j), (W_{k_j,z_j,G_j'}, \gamma_j\circ \psi_j\bar \vf_j\pf'_j))\leq \Const \lambda^{-n} d(W,W').
\]
Taking the sup on the manifolds and test functions and recalling \eqref{ly-gen-1} yields
\[
\| \cL^n h\|_{p,q}\leq C_* \max\{\lambda^{-np}, a^{-1},\lambda^{(p-q)n}\}\|h\|_{p,q}+C_n \|h\|_{0,q},
\]
for some constant $C_*>0$.
To conclude we choose $\bar n$ such that 
\[
\nu^{\min\{p,q-p\}}\geq \left[C_* \max\{\lambda^{-np}, \lambda^{(p-q)n}\}\right]^{1/\bar n},
\]
and then choose $a=\nu^{-\bar n}C_*$. The Lemma follows arguing exactly as at the end of Lemma \ref{lem:ly-h}.
\end{proof}
We leave to the reader the (simple) proof that the unit ball of $\cB^{p,q}$ is weakly compact in $\cB^{0,q}$ for each $q\in (0,\alpha)$ and $p\in (0,q)$. Hence the transfer operator is compact as an operator from $\cB^{p,q}$ to $\cB^{0,q}$. We obtain thus the quasi compactness also in this case. Note however that, due to the low regularity of the map, the essential spectral radius is rather large and it cannot be shrunk by using smaller Banach spaces since on them the Transfer Operator is not well defined.
\begin{rem} The above discussion proves that the essential spectral radius of $\cL$ can be made arbitrarily close to $\lambda^{-\alpha/2}$. The factor $1/2$ in the exponent first appeared in the pioneering work of Kitaev \cite{Ki} and is most likely unavoidable.
\end{rem}

\subsection{ A comment on the discontinuous case}
Another case in which a map has low regularity is when it is only {\em piecewise smooth}. This requires a new idea.

Up to now in the definition of the norms we used manifolds of a fixed, possibly small, size ($\delta$) and the test function were always of compact support. If the map is discontinuous, then $f^{-1}W$ will be cut by the dynamics in several pieces and hence one cannot avoid arbitrarily small manifolds and test functions that are different from zero at the boundary of the manifold.
We are thus forced to include in the set of allowed manifolds $\Sigma$ arbitrarily small manifolds and for $W\in\Sigma$ consider $\vf\in \cC^q(W,\bC)$ rather than $\vf\in \cC^q_0(W,\bC)$. 

This implies that we cannot integrate by part (otherwise we would produce boundary terms that we do not know how to estimate), hence we are limited to $p<1$, even if the map is very regular away from the discontinuities. 

Luckily a second look at Section \ref{sec:lowreg} shows that we never integrated by part, thus we could have worked with $\cC^q(W,\bC)$ as well.\footnote{ Indeed, there was no need to restrict to functions vanishing at the boundary of  the manifold.}
However, a quick inspection to the previous arguments shows that they do not work for arbitrarily small manifolds, as the constants in the Lasota-Yorke inequality depend on $\delta$. It is necessary to treat small manifolds differently.

A possible solution to this problem, first implemented in \cite{DL} and inspired by \cite{Li2}, is to add to the strong norm a term of the form
\[
\sup_{(W, \vf)\in \Omega^q}\frac{1}{|W|^\alpha}\int_W h \vf,
\]
for some $\alpha\in(0,1)$. This means that the integral of $h$ on a small manifolds is small, but not proportional to the volume of the piece, hence $h$ is not necessarily a function and it can have a very wild behaviour on small scales.
\section{ Statistical properties of uniformly Hyperbolic maps}
In section \ref{sec:hyp} we have seen how to extend the functional approach to general Anosov maps. Yet, we did not explained what are the consequences. Here we very briefly discuss what can be obtained by this formalism. We limit the discussion to $\cC^r$ Anosov maps.
\subsection{ Decay of correlations and Limit Theorems}
In Section \ref{sec:norms1} we have seen that $\cL$ is quasi compact, hence it has only finitely many eigenvalues of modulus one.  Moreover, since $\cL$ is a positive operator (it sends positive functions in positive functions) it is possible to prove that the spectrum on the unit circle forms a group under multiplication. In addition, the operator is power bounded and hence it cannot have Jordan blocks, thus the geometric and algebraic multiplicity of the peripheral spectrum are the same. Hence, since one is an eigenvalue, the dimension of the eigenspace associated to the eigenvalue one corresponds to the number of SRB measures. This is quite a bit of information, however the fine structure of the spectrum is not know in general.

In particular, it is not known if Anosov maps always have a unique SRB measure. This depends on global topological properties that are not easily read from the study of the transfer operator. If the map has a unique SRB measure, then there is a dichotomy: either the map is not mixing (there are other eigenvalues, besides one, on the unit circle) or it mixes exponentially fast (one is the only eigenvalue on the unit circle and hence the operator has a spectral gap). 

Accordingly, if the system is mixing, then the rate of mixing is determined by the eigenvalues of the point spectrum of $\cL$. In particular, if an observable belongs to the kernel of the spectral projection of the largest eigenvalues, then it will mix faster.

Without entering in any detail let me conclude by just pointing out that we have now the technology to upgrade all the results of Section \ref{sec:expaning} to the case of uniformly hyperbolic maps. In particular, we can study operators with a smooth potential hence obtain the CLT, Local CLT and Large deviations. Also the perturbation theory of Section \ref{sec:perturb} applies and we can prove stochastic and deterministic stability. Moreover, the slightly more general perturbation theorem in \cite[Section 8]{GL1} implies linear response. In addition, using weighted operators one can construct manifold invariant measures and use the thermodynamic formalism to estimate the Hausdorff dimension of many dynamically relevant sets. There is however an issue that we have not discussed: if one wants to study, e.g., the measure of maximal entropy, then one has to consider a transfer operator with a weight given by the expansion in the stable direction. This, unfortunately, is (in general) only H\"older also for very regular maps. 
Of course one could study such a situation using the norms detailed in Section \ref{sec:lowreg}, however the question remains if it is possible or not to shrink the essential spectrum radius or one has to live with a very large essential spectral radius also for very regular maps. The answer is that the essential spectrum can be shrunk exactly as in Section \ref{sec:norms1}. In order to do so it is however necessary to consider slightly more general Banach spaces, the details can be found, e.g., in \cite{GL2}.

\end{document}